\documentclass{article}
\usepackage{amsmath,amssymb,amsthm, tikz-cd}
\usepackage{graphicx}
\usepackage[shortlabels]{enumitem}
\usepackage[margin=1in]{geometry}
\usepackage{dutchcal}
\usepackage[colorlinks=true,citecolor=blue,linkcolor=blue,urlcolor=blue,pagebackref]{hyperref}
\theoremstyle{plain}
\newtheorem{theorem}{Theorem}[section]
\newtheorem*{theorem*}{Theorem}
\newtheorem{corollary}[theorem]{Corollary}
\newtheorem{lemma}[theorem]{Lemma}
\newtheorem{proposition}[theorem]{Proposition}

\theoremstyle{definition}
\newtheorem{remark}[theorem]{Remark}

\newtheorem{example}[theorem]{Example}

\newtheorem{definition}[theorem]{Definition}

\newtheorem{setup}[theorem]{Setup}

\newlist{enuroman}{enumerate}{1}
\setlist[enuroman]{
	label = (\roman*),
	topsep = 0pt,
	labelindent = \parindent}

\newlist{enualph}{enumerate}{1}
\setlist[enualph]{
	label = (\alph*),
	topsep = 0pt,
	labelindent = \parindent}

\newlist{enuarabic}{enumerate}{3}
\setlist[enuarabic]{
	label = (\arabic*),
	topsep = 0pt,
	labelindent = \parindent}

\newcommand{\Z}{\mathbb{Z}}

\newcommand{\cO}{\mathcal{O}}
\newcommand{\cC}{\mathcal{C}}
\newcommand{\cJ}{\mathcal{J}}

\newcommand{\cA}{\mathcal{A}}

\newcommand{\cP}{\mathcal{P}}
\newcommand{\cD}{\mathcal{D}}
\newcommand{\cU}{\mathcal{U}}

\newcommand{\la}{\langle}
\newcommand{\ra}{\rangle}

\newcommand{\sq}{\subseteq}
\newcommand{\ds}{\dots}

\newcommand{\arrow}{arrow}
\newcommand{\F}{\mathbb{F}}

\newcommand{\fp}{\mathfrak{p}}

\newcommand{\fm}{\mathfrak{m}}

\newcommand{\fn}{\mathfrak{n}}

\renewcommand{\t}[1]{\widetilde{#1}}

\newcommand{\ul}[1]{\underline{#1}}

\DeclareMathOperator{\Hom}{Hom}
\DeclareMathOperator{\End}{End}

\DeclareMathOperator{\Spec}{Spec}

\DeclareMathOperator{\id}{id}

\DeclareMathOperator{\Ass}{Ass}
\DeclareMathOperator{\Supp}{Supp}

\DeclareMathOperator{\mmod}{-mod}

\DeclareMathOperator{\mumod}{-umod}
\DeclareMathOperator{\mfgumod}{-fgumod}

\DeclareMathOperator{\Ten}{Ten}

\renewcommand{\phi}{\varphi}

\DeclareMathAlphabet{\pazocal}{OMS}{zplm}{m}{n}

\newcommand\blfootnote[1]{%
  \begingroup
  \renewcommand\thefootnote{}\footnote{#1}%
  \addtocounter{footnote}{-1}%
  \endgroup
}

\newcommand{\m}{\text{-}}
\DeclareMathOperator{\divmod}{divmod}

\begin{document}

\title{A formalism of $F$-modules for rings with \\ complete local finite $F$-representation type}
\author{Eamon Quinlan-Gallego}
\date{}

\maketitle

\blfootnote{This work was supported by the National Science Foundation under award numbers 1840190 and 2203065.}
\vspace{-30pt}

\begin{abstract}
	We develop a formalism of unit $F$-modules in the style of Lyubeznik and Emerton-Kisin for rings which have finite $F$-representation type after localization and completion at every prime ideal. As applications, we show that if $R$ is such a ring then the iterated local cohomology modules $H^{n_1}_{I_1} \circ \cdots \circ H^{n_s}_{I_s}(R)$ have finitely many associated primes, and that all local cohomology modules $H^n_I(R / gR)$ have closed support when $g$ is a nonzerodivisor on $R$.
\end{abstract}

\section{Introduction}

Let $R$ be a commutative Noetherian ring, $I \sq R$ be an ideal, and $n \geq 0$ be an integer. The local cohomology modules $H^n_I(R)$ are in general not finitely generated, but nonetheless they occasionally exhibit remarkable finiteness properties. In particular, the following two questions have stimulated a lot of work in commutative algebra over the last 30 years. 

\begin{enuarabic}
	\item When is the set of associated primes of $H^n_I(R)$ finite?
	\item When is the support of $H^n_I(R)$ closed? 
\end{enuarabic}

When $R$ is regular (1) is known to hold in great generality; for example, when $R$ contains a field of characteristic $p > 0$ \cite{Huneke-Sharp}, when $R$ is local and contains a field of characteristic zero \cite{Lyu93} 
\cite{Lyubeznik-charfree}, when $R$ is local and unramified of mixed characteristic \cite{Lyubeznik-unramified} \cite{BBLSZ}, and when $R$ is a smooth $\Z$-algebra \cite{BBLSZ}. There are also partial results in the regular ramified case \cite{NB13-2}. 

Question (1) is also known to hold for some classes of singular rings, namely direct summands \cite{NB12} \cite{AMHNB} and rings with finite $F$-representation type \cite{TT08} \cite{HNB17} \cite{DQ20}. There are also results in small dimensions \cite{Marley01} \cite{Robbins12}, for polynomial rings over sufficiently nice bases \cite{NB13} \cite{Robbins14} \cite{Robbins16}, and under strong restrictions on $I$ and $n$ \cite{BF00}, \cite{BRS00}, \cite{DQ20}. On the other hand, there are examples of rings (even with mild singularities) for which (1) does not hold \cite{Singh-counterexample}, \cite{Katzman-counterexample}, \cite{SS-counterexample}.

Note that (2) holds if and only if the module $H^n_I(R)$ has a finite set of {\em minimal} associated primes; therefore (2) holds whenever (1) holds and, in particular, in all the cases described above. It is also known that (2) holds whenever $R = S / gS$ where $S$ is regular and of characteristic $p > 0$ \cite{HNB17} \cite{KZ18} \cite{DQ20}. There are also some results under assumptions on $I$ and $n$ \cite{RS05} \cite{Katzman06} \cite{HKM09}. See \cite{Khashyarmanesh10} and \cite{Lewis22} for more on this direction. Note that there are currently no examples for which (2) does not hold.

In this paper we are interested in the context of positive characteristic, where the Frobenius endomorphism provides an effective tool for tackling these questions. When $R$ is a regular ring of characteristic $p > 0$, $I, I_1, \dots, I_s \sq R$ are ideals, $n, n_1, \dots, n_s \geq 0$ are integers, and $f \in R$ is a nonzerodivisor, we have:
\begin{enualph}
	\item The module $H^n_I(R)$ has a finite number of associated primes \cite{Huneke-Sharp}.
	\item More generally, the module $H^{n_1}_{I_1} \circ \cdots \circ H^{n_s}_{I_s}(R)$ has a finite number of associated primes \cite{Lyubeznik-Fmod}.
	\item The module $H^n_I(R / fR)$ has closed support \cite{HNB17} \cite{KZ18}. 
\end{enualph}
If $R$ is not regular, but instead satisfies the condition of having {\em finite $F$-representation type} (abbreviated FFRT) \cite{SVdB97} (see Section \ref{scn-FFRT}), statement (a) is known to remain true \cite{TT08} \cite{HNB17} \cite{DQ20}, and it is therefore natural to ask whether (b) and (c) remain true in this setting. A major obstacle in this direction is the lack of a well behaved notion of $F$-module in this context, which constitutes the main tool to prove statements (b) and (c) in the regular case. 

In this paper we develop a formalism of $F$-modules that applies to rings with FFRT, and in fact to the wider class of rings that have FFRT after localizing and completion at every prime ideal. We say that such rings have complete local finite $F$-representation type, or CL-FFRT for short. We use this formalism to prove statements (b) and (c) in this setting (and hence also (a)).

After some background in Section \ref{scn-background}, in Section \ref{scn-FFRT} we give some preliminary results on the FFRT and CL-FFRT properties. 

In Section \ref{scn-Frob-desc} we show that rings with CL-FFRT satisfy a weak version (Theorem \ref{thm-Frob-desc}) of the Frobenius descent theorem that holds for regular rings (see \cite{AMBL}). Although this result is not strictly necessary for our applications to local cohomology, some of the techniques used to prove it are. 

In Section \ref{scn-F-mod} we show how, given a ring $R$ with CL-FFRT, one can (after some choices) associate to it a noncommutative ring $\cD^{(0)} \la \Phi \ra$ (which plays the role of $R \la F \ra$ in the regular setting). We then define a special class of left $\cD^{(0)} \la \Phi \ra$-modules known as {\em finitely generated unit}, and we consider the full subcategory that they span in $\cD^{(0)} \la \Phi \ra \mmod$. We then show:
\begin{theorem*}
	\ 
\begin{enuroman}
\item {\upshape (Cor. \ref{cor-fgu-abelian})} The category of finitely generated unit left $\cD^{(0)} \la \Phi \ra$-modules is closed under kernels, cokernels and extensions. This category also satisfies the ascending chain condition.

\item {\upshape (Prop. \ref{prop-R-is-fgu})} The ring $R$ itself is a left $\cD^{(0)} \la \Phi \ra$-module, and as such it is finitely generated unit.

\item {\upshape (Prop. \ref{prop-localization-fg}, \ref{prop-localization-single-elt-unit})} If $M$ is a finitely generated unit left $\cD^{(0)} \la \Phi \ra$-module and $f \in R$ is an element, then $M[1/f]$ is also a finitely generated unit left $\cD^{(0)} \la \Phi \ra$-module.
\end{enuroman}
\end{theorem*}
In Section \ref{scn-app-to-local-coh} we apply these results to local cohomology. Statement (b) is then proved by following Lyubeznik's strategy for regular rings; this gives:
\begin{theorem*}[{Cor. \ref{cor-iterated-finite-ass}}]
	Let $R$ be a ring with {\upshape CL-FFRT}, let $I_1, \dots, I_s \sq R$ be ideals, and $n_1, \dots , n_s \geq 0$ be integers. Then the iterated local cohomology module $H^{i_1}_{I_1} \circ \cdots \circ H^{i_s}_{I_s}(R)$ has finitely many associated primes.
\end{theorem*}
We note that one can also obtain the above theorem with the additional assumptions that $R$ is FFRT, graded and strongly $F$-regular by using a recent machinery of holonomic $\cD$-modules \cite{AMHJNBTW21}.

Statement (c) is then proved by translating the strategy used by Hochster-N\'u\~nez-Betancourt \cite{HNB17} in the regular case (with some technical modifications), giving: 
\begin{theorem*} [{Thm. \ref{thm-closed-supp}}]
	Let $R$ be a ring with {\upshape CL-FFRT}, let $g \in R$ be a nonzerodivisor and $I \sq R$ be an ideal. Then the local cohomology module $H^i_I(R / gR)$ has closed support.
\end{theorem*}
\begin{remark}
	Note that Bhatt and Lurie have developed two notions of $F$-modules that also apply to singular rings $R$ of characteristic $p > 0$ \cite{BL17}, and provide coherent counterparts to constructible \'etale $\F_p$-sheaves on $R$ through Riemann-Hilbert correspondences. The first such notion is that of holonomic $R[F]$-modules; these are required to be perfect, and local cohomology modules do not satisfy this property in general. The second notion is that of derived finitely generated unit complexes; to work with this notion one must pass to the derived category. Note that these two notions are interchanged through a duality functor. It would be interesting to know how the Bhatt-Lurie theory compares to what is developed here. 
\end{remark}

\begin{remark}
	In the regular case, Huneke-Sharp also showed that local cohomology modules has finite Bass numbers in characteristic $p > 0$ \cite{Huneke-Sharp}, which was later generalized by Lyubeznik to the case of iterated local cohomology by using $F$-modules \cite{Lyubeznik-Fmod}. It is hence reasonable to expect that one could give a proof in the CL-FFRT case. Note, however, that a famous example of Hartshorne gives a ring which has FFRT but has a local cohomology module with infinite socle \cite{Hartshorne-inf-socle} (see also \cite{MV-inf-socle}). 
\end{remark}

\paragraph*{Acknowledgements}

I would like to thank Monica Lewis and Anurag Singh for many stimulating conversations on this topic. Thanks as well to Josep \`Alvarez-Montaner, Bhargav Bhatt, Devlin Mallory, Luis N\'u\~nez-Betancourt and Ryo Takahashi for many useful comments on an earlier version of this manuscript. Special thanks to Devlin Mallory for providing Example \ref{ex-devlin} and allowing me to include it here.

\section{Background} \label{scn-background}

\subsection{The Frobenius endomorphism}

Let $p$ be a prime number and $R$ be a commutative ring of characteristic $p > 0$. Then the function $F: R \to R$ given by $F(x) = x^p$ for all $x \in R$ is additive, i.e., satisfies $F(x + y) = F(x) + F(y)$ for all $x, y \in R$. Since it is clearly multiplicative and maps 1 to 1, $F$ is a ring homomorphism. This ring homomorphism $F$ is called the Frobenius endomorphism of $R$. Given an integer $e \geq 0$ we denote by $F^e$ the $e$-th iterate of Frobenius, which is given by $F^e(x) = x^{p^e}$ for all $x \in R$; in particular, $F^0$ is the identity map on $R$. Since $F^e: R \to R$ is a ring homomorphism, its image is a subring of $R$, which we denote by $R^{p^e} := \{x^{p^e} \ | \ x \in R \}$.

Associated to any ring homomorphism one always has a restriction of scalars functor, and the same is true for $F$. However, the fact that $F$ has the same target and source can make this very confusing, and special notation is introduced to deal with this situation. Here we will follow the notation from algebraic geometry, but other notations are also common in the literature; most notably some authors write $M^{1/p^e}$ for what we denote $F^e_* M$.

Given an $R$-module $M$ and an integer $e \geq 0$ we denote by $F^e_* M$ the $R$-module obtained via restriction of scalars via $F^e: R \to R$. In particular, $F^e_* M$ is equal to $M$ as an abelian group. Given an element $u \in M$ we will usually write $F^e_* u$ when we want to emphasize that we view $u$ as an element of $F^e_* M$ and not as an element of $M$. With this notation the $R$-module structure of $F^e_* M$ is given by 
$$x F^e_* u = F^e_* (x^{p^e} u)$$
for all $x \in R$ and $u \in M$. 

\begin{remark} \label{rmk-F_*-notation}
	Here we are making a slight departure from standard terminology in that, for us, $F^e_* M = M$ as sets, whereas most authors would define $F^e_* M$ as the set $\{F^e_* u \ | \ u \in M \}$. Our definition has the pleasant consequence that, given integers $i, j \geq 0$ the set $\Hom_R(F^i_* M, F^j_* M)$ is actually the {\em subset} of $\Hom_\Z(M, M)$ given by
	$$\Hom_R(F^i_* M, F^j_* M) = \{\phi \in \Hom_\Z(M,M) \ | \ \phi(x^{p^j} u) = x^{p^i} \phi(u) \text{ for all } x \in R, u \in M \},$$
	whereas in the other terminology $\Hom_R(F^i_*M , F^j_* M)$ is only naturally identified with this set. This means, for example, that given $\alpha \in \Hom_R(F^i_* M, F^j_* M)$ and $\beta \in \Hom_R(F^a_* M, F^b_* M)$ the composition $\beta \circ \alpha$ is well defined as an element of $\End_\Z(M, M)$; one can then think about what the linearity of $\beta \circ \alpha$ is if necessary. 
\end{remark}

\begin{definition}
	Let $R$ be a Noetherian ring of characteristic $p > 0$. We say that $R$ is $F$-finite whenever the Frobenius endomorphism $F: R \to R$ is module-finite, i.e., whenever $F_* R$ is a finitely generated $R$-module.
\end{definition}

Recall that a composition of module-finite ring homomorphisms is again module-finite. In particular, if $R$ is $F$-finite then $F^e_* R$ is a finitely generated $R$-module for every $e \geq 0$.

Every perfect field is $F$-finite, and the collection of $F$-finite rings is closed under adjoining variables, taking quotients, localization, and completion. This means that, at least when working over a perfect field, most rings that one encounters in algebraic geometry are $F$-finite.  

\subsection{Differential operators in positive characteristic} \label{subscn-diff-ops}

In this subsection we go through Grothendieck's construction of the ring of differential operators, and we give an alternative description of this ring due to Yekutieli in the case of positive characteristic.

Let $k$ be a commutative ring and $R$ be a commutative $k$-algebra. Then the module $\End_k(R)$ has a natural $R \otimes_k R$-module structure, given by
$$\big( (r \otimes s) \phi \big) (x) = r \phi(sx)$$
for all $r, s, x \in R$ and $\phi \in \End_k(R)$. 

Note that there is a natural ``multiplication map" $R \otimes_k R \to R$ given by $[r \otimes s \mapsto rs]$. We let $J_{R|k} \sq R \otimes_k R$ be the kernel of this map. Note that one can give explicit generators for $J_{R | k}$ as
\begin{align*}
	J_{R | k}  	& = (1 \otimes r - r \otimes 1 \ | \ r \in R) \\
		& = (1 \otimes x_i - x_i \otimes 1 \ | \ i \in I),
\end{align*}
where, in the last equality, $\{x_i : i \in I \} \sq R$ are a choice of $k$-algebra generators for $R$. 

\begin{definition}[{\cite[Ch. 16]{EGAIV}}]
	Let $k$ be a commutative ring and $R$ be a commutative $k$-algebra. The ring of $k$-linear differential operators on $R$ is given by
	$$\cD_{R | k} := \{ \phi \in \End_k(R) \ | \ J^n_{R | k} \cdot \phi = 0 \text{ for some } n \gg 0 \}.$$
\end{definition}

The multiplication on $\cD_{R | k}$ is given by composition, i.e., $\cD_{R | k}$ is a subring of $\End_k(R)$. Note that, although it is true, it is not clear from the definition that $\cD_{R | k}$ should be closed under composition.

Suppose now that $R$ is an $F$-finite ring of characteristic $p > 0$ and $k \sq R$ be a perfect field (e.g. one can always take $k = \F_p$). In this situation the ring $D_{R|k}$ of $k$-linear differential operators on $R$ admits another nice description due to Yekutieli \cite{Yek}, which we now explain.

\begin{lemma} \label{lemma-J-cofinal}
	Let $R$ be an $F$-finite ring and $k \sq R$ be a perfect field. The families of ideals $\{J_{R | k}^n \}_{n = 0}^\infty$ and $\{J_{R | k}^{[p^e]}\}_{e = 0}^\infty$ are cofinal.
\end{lemma}

Before we give the proof of Lemma \ref{lemma-J-cofinal}, we note that the statement is clear whenever the ideal $J_{R | k}$ is finitely generated. This happens, for example, when $R$ is essentially of finite type over $k$, since in this case the algebra $R \otimes_k R$ is also essentially of finite type over $k$, and therefore Noetherian. However, we observe that it is not true in general that $R \otimes_k R$ is Noetherian whenever $R$ is Noetherian: borrowing the example of Smith and Van den Bergh \cite{SVdB97}, if $R = k(x_1, x_2, \ds )$ is a rational function field in infinitely many variables then $R \otimes_k R$ is not Noetherian.

\begin{proof}
	For clarity let $J := J_{R|k}$. The inclusion $J^{[p^e]} \sq J^{p^e}$ is clear. Since $R$ is $F$-finite, $R$ is finitely generated as a $R^p$-module, and we fix $R^p$-module generators $x_1, \ds, x_s \in R$. We note that the elements $x_1, \ds, x_s$ also generate $R$ as an algebra (although no longer as a module) over $R^{p^e}$ for every $e \geq 0$. 
	
	We claim that, given $e$, we have an inclusion $J^{[p^e]} \supseteq J^{s(p^e - 1) + 1}$. For this consider the natural algebra homomorphism
	$$R \otimes_k R \to (R \otimes_k R) / J^{[p^e]} = R \otimes_{R^{p^e}} R.$$
	
	This map sends the ideal $J$ into the ideal $\t J = J_{R|R^{p^e}}$, and $\t J$ is generated by the elements $1 \otimes x_1 - x_1 \otimes 1, \dots , 1 \otimes x_s - x_s \otimes 1$. In particular, $\t J^{s(p^e - 1) + 1} \sq \t J^{[p^e]} = 0$. We conclude that the ideal $J^{s(p^e - 1) + 1}$ maps to zero and therefore $J^{s(p^e - 1) + 1} \sq J^{[p^e]}$ as claimed. 
\end{proof}

\begin{proposition} \label{prop-level-filtration}
	Let $R$ be an $F$-finite ring and let $k \sq R$ be a perfect field. Then the ring $\cD_{R|k}$ of $k$-linear differential operators on $R$ is given by
	$$\cD_{R|k} = \bigcup_{e = 0}^\infty \End_{R} (F^e_* R).$$
\end{proposition}
\begin{proof}
	Recall that a $k$-linear endomorphism $\phi \in \End_k(R)$ is a differential operator if and only if it is killed by some power of $J = J_{R|k}$. By Lemma \ref{lemma-J-cofinal}, this is equivalent to being killed by some Frobenius power $J^{[p^e]}$. Note that $J^{[p^e]} = (1 \otimes r^{p^e} - r^{p^e} \otimes 1 : r \in R)$, and that $(1 \otimes r^{p^e} - r^{p^e} \otimes 1) \phi = 0$ if and only if $\phi$ commutes with multiplication by $r^{p^e}$.  We conclude that $J^{[p^e]} \phi = 0$ if and only if $\phi$ is $R^{p^e}$-linear, which gives the result. 
\end{proof}

\begin{remark}
	Recall that with our notation one has an honest inclusion $\End_R(F^e_* R) \sq \End_\Z(R)$ as opposed to an injective map (see Remark \ref{rmk-F_*-notation}). Proposition \ref{prop-level-filtration} also tells us that when $R$ is $F$-finite the ring $\cD_{R|k}$ is independent of the choice of perfect ground field $k$. From this point onwards, given an $F$-finite ring $R$ of characteristic $p$ we will denote by $\cD_{R|\F_p}$ simply by $\cD_R$. 
\end{remark}

\begin{definition}
	Let $R$ be an $F$-finite ring and $e \geq 0$ be an integer. The subring $\End_R(F^e_* R) \sq \cD_R$ is called the ring of differential operators of level $e$, and is denoted by $\cD^e_R := \End_R(F^e_* R)$.
\end{definition}

\begin{proposition} \label{prop-De-noetherian}
	Let $R$ be an $F$-finite ring and $e \geq 0$ be an integer. The ring $\cD^e_R$ is left and right Noetherian.
\end{proposition}
\begin{proof}
	Recall that for us an $F$-finite ring is Noetherian by definition. Since $F^e_* R$ is finitely generated over $R$, we get that $\cD^{e}_R := \End_R(F^e_* R)$ is finitely generated as an $R$-module. Any left ideal $I \sq \cD^{e}_R$ is also an $R$-submodule, and hence $I$ must be finitely generated over $R$, hence finitely generated over $\cD^e_R$. The same is true for a right ideal.
\end{proof}

\subsection{Bimodules} \label{subscn-bimodules}

Suppose $A, B, C$ are rings (not necessarily commutative), that $H$ is an $(A, B)$-bimodule and that $K$ is a $(B, C)$-bimodule. One can then form the tensor product $H \otimes_B C$ by tensoring the right $B$-module structure of $H$ with the left $B$-module structure of $C$. This tensor product has a natural $(A, C)$-bimodule structure, but no natural $B$-module structure in any sense. In particular, when $H$ is a right $B$-module and $K$ is a left $B$-module (equivalenty, $A = C = \Z$) the tensor product $H \otimes_B K$ only has the structure of an abelian group.

As a special case, consider the tensor $H \otimes_A H$ of an $(A, A)$-bimodule $H$. Here it is important to remember that $H \otimes_A H$ is obtained by tensoring the right $A$-module structure of the left copy of $H$ with the left $A$-module structure of the right copy of $H$. The resulting $(A, A)$-bimodule structure is, on the left, given by left multiplication on the left copy of $H$ and, on the right, by right multiplication on the right copy of $H$. 

Since $H \otimes_A H$ is a new $(A, A)$-bimodule, one can iterate this construction to form, for every integer $e \geq 0$, the $e$-fold tensor product
$$H^{\otimes e} := H \otimes_A H \otimes_A \cdots \otimes_A H,$$
where $e$ copies of $H$ appear on the right hand side, with the understanding that $H^{\otimes 0} = A$. The tensor algebra of $H$ is then given by
$$\Ten_A(H) := \bigoplus_{e = 0}^\infty H^{\otimes e}.$$
Note that $\Ten_A(H)$ has a natural ring structure, and that giving a left $\Ten_A(H)$-module $M$ is equivalent to giving a left $A$-module together with an $A$-module map $H \otimes_A M \to M$.

\section{Finite $F$-representation type} \label{scn-FFRT}

\subsection{Complete local case}

Let $R$ be a commutative ring. Choose once and for all a collection $\cJ(R\mmod)$ of finitely generated indecomposable $R$-modules which is exhaustive up to isomorphism, and such that any two distinct elements of $\cJ(R\mmod)$ are not isomorphic. This means that given a finitely generated indecomposable $R$-module $M$ there is a unique $N \in \cJ(R\mmod)$ such that $M \cong N$. We can construct $\cJ(R\mmod)$, for example, by considering the set of all quotients of $R^{\oplus n}$, taking a union over all $n$, and then picking a representative of each isomorphism class. 

Suppose now furthermore that $R$ is Noetherian, local and complete, so that the Krull-Schmidt theorem holds (see \cite[Ch. 1]{LeuschkeWiegand-CM}). This says that given a finitely generated $R$-module $M$ there exist distinct indecomposable $R$-modules $N_1, \dots , N_s \in \cJ(R\mmod)$, integers $a_1, a_2, \dots, a_s$, and an isomorphism
$$M \cong N_1^{\oplus a_1} \oplus N_2^{\oplus a_2} \oplus \cdots N_s^{\oplus a_s},$$
and that the $N_i$ and $a_i$ are uniquely determined by $M$ up to permutation. In this situation, we let $\cJ(M) \sq \cJ(R\mmod)$ denote the subset
$$\cJ(M) = \{ N_1, N_2, \dots , N_s\}.$$
\begin{definition} [\cite{SVdB97}] \label{def-FFRT-completelocal}
	Let $(R, \fm)$ be a complete local $F$-finite ring. We say that $(R, \fm)$ has finite $F$-representation type (FFRT for short) if the set $\bigcup_{e = 0}^\infty \cJ(F^e_* R)$ is finite.
\end{definition}

In other words, $R$ has FFRT whenever there is a finite collection $\{M_1, \dots , M_s \} \sq \cJ(R\mmod)$ such that, for all integers $e \geq 0$, we have an $R$-module isomorphism of the form
$$F^e_* R \cong M_1^{\oplus n_{1,e}} \oplus \cdots \oplus M_s^{\oplus n_{s,e}}.$$
Given a fixed $e \geq 0$, it is not true in general that every $M_i$ occurs as an irreducible component of $F^e_* R$. Indeed, we have
$$\cJ(F^e_* R) = \{ M_i \ | \  n_{i,e} > 0 \}.$$

%,,

\begin{remark} \label{rmk-FFRT-verybasics}
	Let $(R, \fm)$ be a complete local $F$-finite ring, and let $M_1, \dots , M_t$ be finitely generated $R$-modules. For each $i = 1, \dots, t$, take a decomposition of $M_i$ into irreducibles. The direct sum of these decompositions gives a decomposition of $M_1 \oplus \cdots \oplus M_t$ into irreducibles, and we conclude that 
$$\cJ(M_1 \oplus \cdots \oplus M_t) = \cJ(M_1) \cup \cdots \cup \cJ(M_t).$$
\end{remark}

\begin{remark} \label{rmk-FFRT-basics}
	Let $(R, \fm)$, $(S, \fn)$ be complete local $F$-finite rings, let $M$ be a finitely generated $R$-module, and let $G: R\mmod \to S\mmod$ be an additive functor. Suppose $\cJ(M) = \{N_1, \dots , N_s\}$, and take a decomposition $M \cong N_1^{\oplus a_1} \oplus \cdots \oplus N_s^{\oplus a_s}$ into irreducible $R$-modules. Because we assume that $G$ is additive, we have $G(M) \cong G(N_1)^{\oplus a_1} \oplus \cdots \oplus G(N_s)^{\oplus a_s}$, and by Remark \ref{rmk-FFRT-verybasics} we conclude that 
$$\cJ \big(G(M) \big)  = \bigcup_{i = 1}^s \cJ \big( G(N_i) \big).$$
In particular, if $M'$ is another finitely generated $R$-module with $\cJ(M') = \cJ(M)$ then $\cJ(G(M')) = \cJ(G(M))$.  
\end{remark}

%\begin{lemma} \label{lemma-FFRT-verybasics}
%	Let $(R, \fm)$ be a complete local $F$-finite ring, and let $M_1, \dots , M_t$ be finitely generated $R$-modules. Then we have
%	%
%	$$\cJ(M_1 \oplus \cdots \oplus M_t) = \cJ(M_1) \cup \cdots \cup \cJ(M_t).$$
%\end{lemma}
%
%\begin{proof}
%	For each $i = 1, \dots, t$, take a decomposition of $M_i$ into irreducibles. The direct sum of these decompositions gives a decomposition of $M_1 \oplus \cdots \oplus M_t$ into irreducibles, and the statement follows.
%\end{proof}

%\begin{lemma} \label{lemma-FFRT-basics}
%	Let $(R, \fm)$, $(S, \fn)$ be complete local $F$-finite rings, let $M$ be a finitely generated $R$-module, and let $G: R\mmod \to S\mmod$ be an additive functor. If $\cJ(M) = \{N_1, \dots , N_s\}$ then 
%	%
%	$$\cJ \big(G(M) \big)  = \bigcup_{i = 1}^s \cJ \big( G(N_i) \big).$$
%	%
%	In particular, if $M'$ is another finitely generated $R$-module with $\cJ(M) = \cJ(M')$ we get $\cJ(G(M)) = \cJ(G(M'))$.
%\end{lemma}
%%
%\begin{proof}
%	Take a decomposition $M = N_1^{\oplus a_1} \oplus \cdots \oplus N_s^{\oplus a_s}$ into irreducible $R$-modules. Since $G$ is additive, we have $G(M) = G(N_1)^{\oplus a_1} \oplus \cdots \oplus G(N_s)^{\oplus a_s}$, and the statement follows from Lemma \ref{lemma-FFRT-verybasics}.
%\end{proof}

\begin{proposition} \label{prop-FFRT-iff-FFRTcouple}
	Let $(R, \fm)$ be a complete local $F$-finite ring. Then the following are equivalent:
	\begin{enualph}
		\item The ring $(R, \fm)$ has {\upshape FFRT}.
		\item There exists integers $a \geq 0$ and $b > 0$ such that $\cJ(F^a_* R) = \cJ(F^{a+b}_* R)$. 
	\end{enualph}
	Moreover, whenever (b) occurs we have
	$$\cJ(F^a_* R) = \cJ(F^{a+b}_* R) = \cJ(F^{a + 2b}_* R) = \cdots $$	
\end{proposition}
\begin{proof}
  	Suppose that $R$ has FFRT. Let $\cP$ denote the power set of $\bigcup_{e = 0}^\infty \cJ(F^e_* R)$. Since $\cP$ is a finite set, the sequence $(\cJ(F^e_* R))_{e = 0}^\infty \sq \cP$ must eventually have a repetition, and hence there are integers $a \geq 0$ and $b > 0$ such that $\cJ(F^a_* R) = \cJ(F^{a+b}_* R)$. This shows that (a) implies (b).

	Note that, by Remark \ref{rmk-FFRT-basics}, whenever $\cJ(F^a_* R) = \cJ(F^{a+b}_* R)$ for some $a \geq 0$ and $b > 0$, for every $e \geq 1$ we get
	$$\cJ(F^{a+2b}_* R) = \cJ(F^{b}_* F^{a+b}_* R) = \cJ(F^{b}_*  F^a_* R) = F^{a+b}_* R,$$
	and the last statement of the proposition follows by induction.
	
	Let's now prove that (b) implies (a). Suppose that for some integers $a \geq 0$ and $b > 0$ we have $\cJ(F^{a+eb}_* R) = \cJ(F^a_* R)$ for every integer $e \geq 0$. Given $j \geq a + b$, we can find integers $e \geq 1$ and $0 \leq r < b$ such that $j = a + eb + r$. We then have
	$$\cJ(F^j_* R) = \cJ(F^r_* F^{a+eb}_* R) =  \cJ(F^r_* F^a_* R) = \cJ(F^{a+r}_* R),$$
	where the second equality uses Remark \ref{rmk-FFRT-basics}. This proves that
	$$\bigcup_{e  = 0}^{a + b - 1} \cJ(F^e_* R) = \bigcup_{j = 0}^\infty \cJ(F^j_* R).$$
	Since the left-hand side is clearly finite, this shows that $R$ has FFRT. 
\end{proof}

\begin{remark} \label{rmk-FFRT-couple-localization}
Let $(R, \fm)$ be a complete local $F$-finite ring, and suppose that $R$ has {\upshape FFRT} with
$$\cJ(F^a_* R) = \cJ(F^{a+b}_* R) = \cJ(F^{a + 2b}_* R) = \cdots.$$	
Let $\cJ(F^a_* R) = \{N_1, \dots , N_s \}$, so that $\cJ(F^{a+eb}_* R) =  \{N_1, \dots , N_s \}$ for every $e \geq 0$, and fix a prime ideal $\fp \sq R$. By applying Remark \ref{rmk-FFRT-basics} to the functor $G(M) = \widehat{M}_\fp$ we get, for all $e \geq 0$,
$$\cJ(F^{a+eb}_* (\widehat{R}_\fp)) = \bigcup_{i = 1}^s \cJ( \widehat{N}_{i, \fp}) = \cJ(F^a_* \widehat{R}_\fp),$$
and therefore
$$\cJ(F^a_* \widehat{R}_\fp) = \cJ(F^{a+b}_* \widehat{R}_\fp) = \cJ(F^{a + 2b}_* \widehat{R}_\fp) = \cdots$$
In particular, $\widehat{R}_\fp$ has {\upshape FFRT} by Proposition \ref{prop-FFRT-iff-FFRTcouple}. 
\end{remark}

%
%\begin{lemma} \label{lemma-FFRT-couple-localization}
%	Let $(R, \fm)$ be a complete local $F$-finite ring, and suppose that $R$ has {\upshape FFRT} with
%	%
%	$$\cJ(F^a_* R) = \cJ(F^{a+b}_* R) = \cJ(F^{a + 2b}_* R) = \cdots.$$	
%	%
%	Then for every prime ideal $\fp \sq R$ we have 
%	%
%	$$\cJ(F^a_* \widehat{R}_\fp) = \cJ(F^{a+b}_* \widehat{R}_\fp) = \cJ(F^{a + 2b}_* \widehat{R}_\fp) = \cdots$$
%	%
%		and, in particular, $\widehat{R}_\fp$ has {\upshape FFRT}. 
%\end{lemma}
%%
%\begin{proof}
%	Let $\cJ(F^a_* R) = \{N_1, \dots , N_s \}$, so that $\cJ(F^{a+eb}_* R) =  \{N_1, \dots , N_s \}$ for every $e \geq 0$. By applying Remark \ref{rmk-FFRT-basics} to the functor $G(M) = \widehat{M}_\fp$ we get, for all $e \geq 0$,
%	%
%$$\cJ(F^{a+eb}_* (\widehat{R}_\fp)) = \bigcup_{i = 1}^s \cJ( \widehat{N}_{i, \fp}) = \cJ(F^a_* \widehat{R}_\fp),$$
%	%
%	which proves the claim.
%\end{proof}

\subsection{Global case}

We now introduce a global notion of the FFRT property by requiring that completions at all prime ideals have FFRT. We compare this notion with another notion that is present in the literature. Our main result here will be that the integers $a \geq 0$ and $b > 0$ satisfying the conclusion of Proposition \ref{prop-FFRT-iff-FFRTcouple} can be chosen globally.

\begin{definition} \label{def-FFRT-global}
	Let $R$ be an $F$-finite ring. We say that $R$ has complete local finite $F$-representation type (CL-FFRT for short) whenever for all prime (equivalently, maximal) ideals $\fp \sq R$ the ring $\widehat{R_\fp}$ has FFRT.
\end{definition}

\begin{remark}
	The fact that it suffices to check at maximal ideals follows from Remark \ref{rmk-FFRT-couple-localization}. In particular, if $R$ is local to begin with then $R$ has CL-FFRT if and only if $\widehat{R}$ has FFRT. Notably, for complete local rings the notions of CL-FFRT and FFRT are equivalent. 
\end{remark}

\begin{remark}
	Away from the complete local case, Y. Yao's notion of FFRT is much more common in the literature \cite{Yao-FFRT}. He defines $R$ to have FFRT whenever whenever there exist finitely generated $R$-modules $M_1, \dots M_s$ such that $F^e_* R$ is a direct sum of these modules for every $e \geq 0$. We next show that FFRT in this sense implies CL-FFRT, and we give an example of D. Mallory to illustrate that these two notions are indeed different.
\end{remark}
	
\begin{proposition} \label{prop-weak-vs-strong-FFRT}
	Let $R$ be an $F$-finite ring. Suppose there exist $R$-modules $M_1, M_2, \dots , M_s$ such that for every $e \geq 0$ we have an isomorphism
	$$F^e_* R \cong M_1^{\oplus n_{1, e}} \oplus M_2^{\oplus n_{2, e}} \oplus \cdots \oplus M_s^{\oplus n_{s, e}}$$
	for some integers $n_{i, e} \geq 0$. Then $R$ has {\upshape CL-FFRT}.
\end{proposition}
\begin{proof}
Fix a prime ideal $\fp \sq R$ and an integer $e \geq 0$. We then have $F^e_* \widehat{R}_\fp \cong \widehat{M}_{1, \fp}^{\oplus n_{1,e}} \oplus \cdots \oplus \widehat{M}_{s, \fp}^{\oplus n_{s, e}}$. By Remark \ref{rmk-FFRT-verybasics}, we get
$$\cJ (F^e_* \widehat{R}_\fp) \sq \cJ(\widehat{M}_{1, \fp}) \cup \cdots \cup \cJ(\widehat{M}_{s, \fp}).$$
	Since the right hand side is finite and does not depend on $e$ we conclude that $\widehat{R}_\fp$ has FFRT.
\end{proof}

We would like to thank Devlin Mallory for allowing us to include this example of his here.

\begin{example} [D. Mallory] \label{ex-devlin}
	Let $X$ be an smooth ordinary elliptic curve of characteristic $p > 0$, and let $x \in X$ be a closed point; then $U := X \setminus \{x\}$ is affine. It is known that $F^e_* \cO_X$ decomposes as $F^e_* \cO_X = \bigoplus_{i = 0}^{p^e} \mathcal{L_i}$ where the $\mathcal{L}_i$ are pairwise non-isomorphic $p^e$-torsion line bundles; Smith and Van den Bergh use this to show that the homogeneous coordinate ring of $X$ does not have FFRT in the graded sense \cite{SVdB97}. Set $R := \Gamma(U, \cO_X)$; note that $F^e_* R = \bigoplus_{i = 0}^{p^e} \Gamma(U, \mathcal L_i)$, and we claim that the $\Gamma(U, \mathcal L_i)$ are pairwise non-isomorphic. Indeed, if we had $\mathcal L_i |_U \cong  \mathcal L_j |_U$ then $\mathcal L_i \cong \mathcal L_j(n \{x\})$ for some $n \in \Z$. Since both $\mathcal L_i$, $\mathcal L_j$ are torsion they have degree zero, and this forces $n = 0$, and thus $\mathcal L_i \cong \mathcal L_j$, and hence $i = j$. This shows that $R$ does not have FFRT in the sense of Yao despite being regular. However, note that every regular $F$-finite ring is CL-FFRT.
\end{example}

\begin{remark}
	Proposition \ref{prop-weak-vs-strong-FFRT} also shows that if $R = \bigoplus_{n = 0}^\infty R_n$ is finitely generated graded over an $F$-finite field $R_0 = k$, and $R$ is FFRT in the graded sense of Smith and Van den Bergh \cite{SVdB97}, then $R$ has CL-FFRT. 
\end{remark}

For examples of classes of rings which have the FFRT property, we refer the reader to \cite[Ex. 1.3]{TT08}, \cite{Hara15} \cite{Shibuta17} \cite{Alhazmy-Katzman} \cite{RSVdB19} \cite{RSVdB22}. On the other hand, proving that a given ring does {\em not} have the FFRT property is a tough problem, but there are also results in this direction. For example, we know that rings for which some local cohomology module has infinitely many associated primes cannot be FFRT \cite{TT08} \cite{HNB17} \cite{DQ20} (such examples are constructed in  \cite{Katzman-counterexample} \cite{SS-counterexample}). Using other techniques many other classes of rings have been shown to not have the FFRT property; see \cite{Hara-Ohkawa} \cite{Mallory22}.

Suppose that $R$ has CL-FFRT. By Proposition \ref{prop-FFRT-iff-FFRTcouple}, for every prime ideal $\fp \sq R$ there exist integers $a \geq 0$ and $b > 0$ such that
$$\cJ(F^a_* \widehat{R}_\fp) = \cJ(F^{a+b}_* \widehat{R}_\fp) = \cJ(F^{a+2b}_* \widehat R_\fp) = \cdots .$$
It turns out that these integers $a$ and $b$ can be chosen independently of $\fp$, and we prove this in the remainder of this section. We begin with an observation in the complete local case.

\begin{lemma} \label{lemma-irr-subset-equiv}
	Let $(R, \fm)$ be a complete local ring, let $M$ and $N$ be finitely generated $R$-modules, and let $E(M) := \End_R(M)$, $E(N) := \End_R(N)$. Consider the map 
	$$\Psi: \Hom_R(M,N) \otimes_{E(M)} \Hom_R(N,M) \to E(N)$$
	given by $\Psi(\phi_2 \otimes \phi_1) = \phi_2 \circ \phi_1$. The following are equivalent:
	\begin{enualph}
	\item We have $\cJ(N) \sq \cJ(M)$.
	\item The map $\Psi$ is an isomorphism.
	\item The image of $\Psi$ contains $\id_N$.
	\end{enualph}
\end{lemma}
\begin{proof}
	If (a) holds there is some integer $n \geq 1$ such that $N$ is a direct summand of $M^{\oplus n}$, and one can find maps $\alpha_i: N \to M$ and $\beta_i: M \to N$ ($i = 1, \dots , n)$ such that $\id_N = \sum_i \beta_i \circ \alpha_i$. One then checks that $[\psi \mapsto \sum_i \beta_i \otimes \alpha_i \circ \psi]$ provides a two-sided inverse for $\Psi$, and hence (b) holds.

	Statement (b) implies (c) trivially. Suppose now that (c) holds, which means that there exist $\beta_i: N \to M$ and $\alpha_i: M \to N$ ($i = 1, \dots , n)$ such that $\sum_i \beta_i \circ \alpha_i = \Psi(\sum_i \beta_i \otimes \alpha_i) = \id_N$. Consider then the maps $\beta: N \to M^{\oplus n}$ and $\alpha: M^{\oplus n} \to N$ given by $\beta = \beta_1 \oplus \cdots \oplus \beta_n$ and $\alpha = \alpha_1 + \cdots + \alpha_n$. We then get $\beta \circ \alpha = \id_N$, and therefore $N$ is a direct summand of $M^{\oplus n}$, giving (a).
\end{proof}

Let $R$ be an $F$-finite ring. Given integers $a \geq 0$ and $b > 0$, we define the subset $\cU_{a,b} \sq \Spec(R)$ by
\begin{align*}
	\cU_{a,b} & := \{ \fp \in \Spec(R) \ | \ \cJ(F^a_* \widehat{R}_\fp) = \cJ(F^{a+b}_* \widehat{R}_\fp)\} \\
		& = \{ \fp \in \Spec(R) \ | \ \cJ(F^a_* \widehat{R}_\fp) = \cJ(F^{a+b}_* \widehat{R}_\fp) = \cJ(F^{a+2b}_* \widehat{R}_\fp) = \cdots \},
\end{align*}
where the second description follows from Proposition \ref{prop-FFRT-iff-FFRTcouple}.

%\begin{lemma} \label{lemma-U-containment}
%	Let $R$ be an $F$-finite ring. The subsets $\cU_{a,b} \sq \Spec(R)$ defined above satisfy the following containments:
%	\begin{enuroman}
%	\item We have $\cU_{a, b}  \sq \cU_{a, 2b} \sq \cU_{a, 3b} \sq \cdots$.
%	\item We have $\cU_{a,b} \sq \cU_{a+1, b} \sq \cU_{a+2, b} \sq \cdots $. 
%	\end{enuroman}
%\end{lemma}
%
%\begin{proof}
%	Statement (i) follows from the description given above. For (ii), suppose $\fp \in \cU_{a,b}$, i.e., we have $\cJ(F^a_* \widehat R_\fp) = \cJ(F^{a+b}_* \widehat R_\fp)$. By Remark \ref{rmk-FFRT-basics}, we then get $\cJ(F^{a+1}_* \widehat R_\fp) = \cJ(F_* F^a_* \widehat R_\fp) = \cJ(F_* F^{a+b}_* \widehat R_\fp) = \cJ(F^{a+b+1}_* \widehat R_\fp)$. 
%\end{proof}

%\begin{lemma} \label{lemma-U-open}
%	Let $R$ be an $F$-finite ring and fix integers $a \geq 0$ and $b > 0$. The subset $\cU_{a,b} \sq \Spec(R)$ is open in the Zariski topology.
%\end{lemma}
%%
%\begin{proof}
%	By Lemma \ref{lemma-irr-subset-equiv}, $\cU_{a,b}$ is the locus where both morphisms
%	\begin{align*}
%		\Hom_R(F^a_* R, F^{a+b}_* R) \otimes_{\cD^{a}_R} \Hom_R(F^{a+b}_* R, F^{a}_* R) & \longrightarrow \cD^{a+b}_R \\
%		\Hom_R(F^{a+b}_* R, F^{a}_* R) \otimes_{\cD^{a+b}_R} \Hom_R(F^{a}_* R, F^{a+b}_* R) & \longrightarrow \cD^{a}_R
%	\end{align*}
%	are isomorphisms.
%\end{proof}

\begin{proposition}
	Let $R$ be an $F$-finite ring. Then the locus
	$$\{\fp \in \Spec(R) \ | \ \widehat{R}_\fp \text{ \upshape has FFRT } \}$$
	is open in the Zariski topology.
\end{proposition}
\begin{proof}
	By Proposition \ref{prop-FFRT-iff-FFRTcouple}, the set given equals $\bigcup_{a \geq 0, b > 0} \cU_{a,b}$, so it suffices to show that each $\cU_{a,b}$ is open. For this observe that, by Lemma \ref{lemma-irr-subset-equiv}, $\cU_{a,b}$ is the locus where both morphisms
	\begin{align*}
		\Hom_R(F^a_* R, F^{a+b}_* R) \otimes_{\cD^{a}_R} \Hom_R(F^{a+b}_* R, F^{a}_* R) & \longrightarrow \cD^{a+b}_R \\
		\Hom_R(F^{a+b}_* R, F^{a}_* R) \otimes_{\cD^{a+b}_R} \Hom_R(F^{a}_* R, F^{a+b}_* R) & \longrightarrow \cD^{a}_R
	\end{align*}
	are isomorphisms.
\end{proof}

\begin{remark}
	K. Alhazmy has shown that the locus of $\fp \in \Spec(R)$ for which $R_\fp$ has FFRT in the sense of Yao is also open \cite{Alhazmy-FFRT}.
\end{remark}

\begin{proposition} \label{prop-exists-FFRT-tuple}
	Let $R$ be a ring with {\upshape CL-FFRT}. Then there exist integers $a \geq 0$ and $b > 0$ such that $\cU_{a,b} = \Spec(R)$ or, in other words, such that for all prime ideals $\fp \sq R$ we have
	$$\cJ(F^a_* \widehat{R}_\fp) = \cJ(F^{a+b}_* \widehat{R}_\fp) = \cJ(F^{a+2b}_* \widehat{R}_\fp) = \cdots$$
\end{proposition}
\begin{proof}
	Since $R$ has CL-FFRT, we know that $\Spec(R) = \bigcup_{i \geq 0, j > 0} \cU_{i, j}$. By quasicompactness this open cover has a finite subcover. To finish the proof we show that this union is directed.

We begin by observing that the subsets $\cU_{a,b} \sq \Spec(R)$ defined above satisfy the following containments:
	\begin{enuroman}
	\item We have $\cU_{a, b}  \sq \cU_{a, 2b} \sq \cU_{a, 3b} \sq \cdots$.
	\item We have $\cU_{a,b} \sq \cU_{a+1, b} \sq \cU_{a+2, b} \sq \cdots $. 
	\end{enuroman}
	Statement (i) follows from the description of $\cU_{a,b}$ given above. For (ii), suppose $\fp \in \cU_{a,b}$, i.e., we have $\cJ(F^a_* \widehat R_\fp) = \cJ(F^{a+b}_* \widehat R_\fp)$. By Remark \ref{rmk-FFRT-basics}, we then get $\cJ(F^{a+1}_* \widehat R_\fp) = \cJ(F_* F^a_* \widehat R_\fp) = \cJ(F_* F^{a+b}_* \widehat R_\fp) = \cJ(F^{a+b+1}_* \widehat R_\fp)$.

With these containments we see that, given integers $i, k \geq 0$ and $j, l > 0$, we obtain
	\[\cU_{i,j} \cup \cU_{k,l} \sq \cU_{\max \{ i,k \}, j} \cup \cU_{\max \{i,k\}, l} \sq \cU_{\max \{i,k\}, jl} \qedhere \]
\end{proof}

\section{Frobenius descent} \label{scn-Frob-desc}

Let $R$ be a regular $F$-finite ring and, as usual, we denote by $\mathcal D$ its ring of differential operators. In \cite{Katz70} Katz observed that, given an $R$-module $M$, its Frobenius pullback $F^* M$ is endowed with an integrable connection with vanishing $p$-curvature; in the language of this paper, this means that $F^* M$ has a natural $\mathcal D^1$-module structure. Moreover, Katz showed that the functor $F^*: R\mmod \to \cD^1  \mmod$ is an equivalence of categories.

With our notation and definitions these statements can be proved efficiently. Roughly speaking, we have $F^* M = F_* R \otimes_R M$ and therefore $\cD^1  = \End_R(F_* R)$ acts on $F^*M$ by acting on the left of the tensor. Moreover, since $R$ is regular and $F$-finite by assumption, we know that $F_* R$ is a finitely generated projective $R$-module; the fact that $F^*: R\mmod \to \cD^1  \mmod$ is an equivalence of categories is then a standard fact from Morita theory (or see below, as this section has a self-contained proof).

This point of view also allows for some generalization. Indeed, under the same assumptions we know that $F^e_* R$ is finitely generated projective for every $e \geq 0$, and therefore one gets equivalences of categories $F^{e*}: R \mmod \to \cD^e  \mmod$ for every $e \geq 0$. Moreover, these equivalences fit into a chain
\begin{center}
	\begin{tikzcd}
		R \mmod \arrow[r, "F^{*}"] \arrow[r, "\sim", swap] & \cD^{ 1 }\mmod \arrow[r, "F^{*}"] \arrow[r, "\sim", swap]  & \cD^{ 2 }\mmod \arrow[r, "F^{*}"] \arrow[r, "\sim", swap] & \cdots \arrow[r, "F^{*}"]\arrow[r, "\sim", swap]  &  \cD^{ e }\mmod \arrow[r, "F^{*}"] \arrow[r, "\sim", swap] & \cdots.
	\end{tikzcd}
\end{center}
These facts (informally known as Frobenius descent) prove crucial the study of rings of differential operators for regular rings in positive characteristic; see \cite{Chase-homdim} \cite{SmithSP-diffOps} \cite{SmithSP-gdim} \cite{Haastert-invdir} \cite{Bogvad} \cite{Blickle-thesis} \cite{AMBL}. See also \cite{BerthelotII} \cite{Berthelot-divided} for Frobenius descent in the context of arithmetic $\mathcal D$-modules.

In this section we show that rings with CL-FFRT satisfy a weak version of the Frobenius descent theorem described above. Some of the results in this section follow from standard facts in Morita theory, but we try to provide a full exposition to keep this article self-contained.
\begin{setup} \label{setup-FFRT}
	We fix the following notation: $R$ will be an $F$-finite ring with CL-FFRT. We fix integers $a \geq 0$ and $b > 0$ such that, for every prime ideal $\fp \sq R$, the $\widehat{R_\fp}$-modules $F^a_* \widehat{R_\fp}, F^{a+b}_* \widehat{R_\fp}, F^{a+2b}_* \widehat{R_\fp}, \dots$ have the same indecomposables, i.e.,
	$$\cJ(F^{a}_* \widehat{R_\fp}) = \cJ(F^{a+b}_* \widehat{R_\fp}) = \cJ(F^{a+2b}_* \widehat{R_\fp}) = \cdots$$
	(note that the existence of such integers is guaranteed by Proposition \ref{prop-exists-FFRT-tuple}). With these integers fixed, for every $e \geq 0$ we will denote $\cD^{(e)} := \cD^{a + eb}_R = \End_R(F^{a+eb}_* R)$ and $F^{(e)} := F^{a + eb}$. In particular, if $M$ is an $R$-module then $F^{(e)}_* M = F^{a+eb}_* M$. 
\end{setup}

Note that when $R$ is regular one can pick $a = 0$, $b = 1$ by Kunz's theorem. By making this choice in what follows (i.e., $\cD^{(e)} = \cD^e$ and $F^{(e)} = F^e$), one recovers the theory of Frobenius descent in the regular case as described above.

If $R$ is $F$-split the Frobenius map $F^e_* R \to F^{e+1}_* R$ is a split inclusion, and hence it remains so after localization and completion. We conclude that every indecomposable summand of $F^e_* \widehat{R_\fp}$ appears in $F^{e+1}_* \widehat{R_\fp}$, meaning that when $R$ is $F$-split and CL-FFRT one can always choose $b = 1$. 

\begin{definition} \label{def-FFRT-F-pullback}
	Let $R$ be a ring with CL-FFRT, and fix notation as in Setup \ref{setup-FFRT}. Let $s \geq 0$ be an integer. We define the functor $\Phi^{s*}: \  \cD^{(0)}\mmod \longrightarrow \cD^{(s)}\mmod$ by $\Phi^{s*} := \Hom_R(F^{(0)}_* R, F^{(s)}_* R) \otimes_{\cD^{(0)}} (-)$. When $s = 1$ we write this simply as $\Phi^*$. Note that this depends on the choices of integers $a$ and $b$. 
\end{definition}
We intend for this notation to be suggestive. When $R$ is regular (and $a = 0, b = 1$), the functor $\Phi^{s*}$ is $F^{s*}$, where $F^s: R \to F^s_* R$ is the $s$-times iterated Frobenius morphism (note that $F^*$ is sometimes called the Peskine-Szpiro functor). But, as far as we know, in general there is no morphism $\Phi$ for which $\Phi^{s*}$ is the pullback via $\Phi$. Nonetheless we will soon see that we do have $(\Phi^*)^s = \Phi^{s*}$ in an appropriate sense. 

\begin{lemma} \label{lemma-compo-isom-2}
	Let $i, j, k, m \geq 0$ be integers. We have a $(\cD^{(i)}, \cD^{(i + k)})$-bimodule isomorphism
	$$\Hom_R(F^{(i)}_* R, F^{(j)}_* R) \otimes_{\cD^{(i)}} \Hom_R(F^{(i+k)}_*, F^{(i + m)}_* R) \xrightarrow{\sim} \Hom_R (F^{(i+k)}_* R, F^{(j+m)}_* R)$$
	given by $[\phi_2 \otimes \phi_1 \mapsto \phi_2 \circ \phi_1]$ (see Remark \ref{rmk-F_*-notation}).
\end{lemma}
\begin{proof}
	The statement can be checked after localization and completion, so we may assume that $R$ is complete local and that all of the modules $F^{(e)}_* R$ have the same indecomposable summands. 

	We then construct a two sided inverse. We know there is some integer $n \geq 0$ such that $F^{(j)}_* R$ is a direct summand of $(F^{(i)}_* R)^{\oplus n}$, and hence there are $R$-linear maps $\alpha_t : F^{(j)}_* R \to  F^{(i)}_* R$ and $\beta_t : F^{(i)}_* R \to F^{(j)}_* R$, $t = 1, \dots n$, such that the identity map on $F^{(j)}_* R$ can be written as $\id = \sum_t \beta_t \circ \alpha_t$.  With this notation, the two-sided inverse is given by $[\psi \mapsto \sum_t \beta_t \otimes  \alpha_t \circ \psi]$. 
\end{proof}

\begin{proposition} \label{prop-higher-Frob-pullback}
	Let $R$ be a ring with {\upshape CL-FFRT}, and fix notation as in Setup \ref{setup-FFRT}. If $M$ is a left $\cD^{(e)}$-module then $\Phi^{s*}(M)$ admits a natural $\cD^{(e+s)}$-module structure. Moreover, the induced functor $\cD^{(e)}\mmod \to \cD^{(e+s)}\mmod$ is given by 
	$$\Hom_R(F^{(e)}_* R, F^{(e+s)}_* R) \otimes_{\cD^{(e)}} (-).$$
\end{proposition}
\begin{proof}
	We have the following isomorphisms of left $\cD^{(s)}$-modules
	\begin{align*}
		\Phi^{s*} (M) & = \Hom_R(F^{(0)}_* R, F^{(s)}_* R) \otimes_{\cD^{(0)}} M \\	
			& \cong \Hom_R(F^{(0)}_* R, F^{(s)}_* R) \otimes_{\cD^{(0)}} \cD^{(e)} \otimes_{\cD^{(e)}} M \\
			& \cong \Hom_R(F^{(e)}_*R, F^{(e+s)}_* R) \otimes_{\cD^{(e)}} M,
	\end{align*}
	where the last one follows from Lemma \ref{lemma-compo-isom-2}. Since $\Hom_R(F^{(e)}_*R, F^{(e+s)}_* R) \otimes_{\cD^{(e)}} M$ admits a natural left $\cD^{(e+s)}$-module structure, the first claim follows. The second statement also follows from the chain of isomorphisms above. 
\end{proof}

We will abuse notation and denote the induced functor from Proposition \ref{prop-higher-Frob-pullback} also by $\Phi^{s*}: \cD^{(e)}\mmod \to \cD^{(e+s)}\mmod$, the value of $e$ being clear from context. With this notation, setting $s = 1$, we have built a collection of functors as follows:
\begin{center}
	\begin{tikzcd}
		\cD^{(0)}\mmod \arrow[r, "\Phi^{*}"] & \cD^{(1)}\mmod \arrow[r, "\Phi^{*}"] & \cD^{(2)}\mmod \arrow[r, "\Phi^{*}"] & \cdots \arrow[r, "\Phi^{*}"] &  \cD^{(e)}\mmod \arrow[r, "\Phi^{*}"] & \cdots 
	\end{tikzcd}
\end{center}
Soon we will show all these functors are equivalences of categories. We begin by showing that composition behaves well.
\begin{lemma} \label{lemma-compo-isom}
	Let $S$ be a commutative ring and let $M_1, M_2, M_3$ be $S$-modules, and let $E_2 := \End_S(M_2)$. Suppose that there is some integer $n \geq 0$ such that $M_1$ or $M_3$ is a direct summand of $M_2^{\oplus n}$. Then the natural map
	$$\Hom_S(M_2, M_3) \otimes_{E_2} \Hom_S(M_1, M_2) \longrightarrow \Hom_S(M_1, M_3)$$
	given by $[\phi_2 \otimes \phi_1 \mapsto \phi_2 \circ \phi_1]$ is an isomorphism. 
\end{lemma}
\begin{proof}
	We construct a two-sided inverse in each case. In the case where $M_3$ is a direct summand of $M_2^{\oplus n}$ we can find maps $\alpha_i: M_3 \to M_2$ and $\beta_i: M_2 \to M_3$ such that the identity map on $M_3$ can be written as $\id_{M_3} = \sum_i \beta_i \circ \alpha_i$, and we claim that $[\psi \mapsto \sum_i \beta_i \otimes \alpha_i \circ \psi]$ defines a two sided inverse. 
	
	It is a right inverse because for all $\psi \in \Hom_S(M_1, M_3)$ we have
	\begin{align*}
		\sum_i \beta_i \circ \alpha_i \circ \psi & = \big( \sum_i \beta_i \circ \alpha_i \big) \circ \psi \\
		&  = \psi.\\
		\intertext{It is a left inverse because for all $\phi_1 \in \Hom_S(M_1, M_2)$ and all $\phi_2 \in \Hom_S(M_2, M_3)$ we have}
		\sum_i \beta_i \otimes \alpha_i \circ \phi_2 \circ \phi_1 & = \sum_i \beta_i \circ \alpha_i \circ \phi_2 \otimes \phi_1 \\ 
		& = \big( \sum_i \beta_i \circ \alpha_i \big) \circ \phi_2 \otimes \phi_1 \\
		& = \phi_2 \otimes \phi_1,
	\end{align*}
	where the first equality in the chain above follows because $\alpha_i \circ \phi_2 \in E_2$, and thus it commutes with the tensor.
	
	The case where $M_1$ is a direct summand of $M_2^{\oplus n}$ is similar: one expresses the identity map on $M_1$ as $M_1 = \sum_i \beta_i \circ \alpha_i$ where $\alpha_i: M_1 \to M_2$ and $\beta_i: M_2 \to M_1$ are $S$-linear maps, and the two sided inverse is given by $[\psi \mapsto \sum_i \psi \circ \beta_i \otimes \alpha_i]$. 
\end{proof}

\begin{proposition} \label{prop-Frob-pullback-compo}
	Let $R$ be a ring with {\upshape CL-FFRT}, fix notation as in Setup \ref{setup-FFRT}, and let $e, s, t \geq 0$ be nonnegative integers. The functor $\Phi^{(s + t)*}: \cD^{(e)}\mmod \to \cD^{(e+s+t)}\mmod$ is naturally equivalent to the composition of functors $\Phi^{t*} \circ \Phi^{s*}: \cD^{(e)}\mmod \to \cD^{(e+s)}\mmod \to \cD^{(e+s+t)}\mmod$. 
\end{proposition}
\begin{proof}
	Composition gives us a natural morphism 
	$$\Hom_R(F^{(e+s)}_* R, F^{(e+s+t)}_* R) \otimes_{\cD^{(e+s)}} \Hom_R(F^{(e)}_* R, F^{(e+s)}_* R)  \to  \Hom_R(F^{(e)}_* R, F^{(e+s+t)}_* R)$$
	of $(\cD^{(e+s+t)}, \cD^{(e)})$-bimodules. We claim that this is an isomorphism. 
	
	If we assume the claim, note that the functor $\Phi^{(s + t)*}$ is given by tensoring with the second bimodule, whereas the composition $\Phi^{t*} \circ \Phi^{s*}$ is given by tensoring with the first bimodule (see Proposition \ref{prop-higher-Frob-pullback}). Hence the statement follows once we prove the claim.  
	
	Finally, note that the claim can be proved after localization and completion, so we may assume that $R$ is complete local and that all of the $R$-modules $F^{(i)}_* R$ have the same indecomposable summands, and thus the claim follows from Lemma \ref{lemma-compo-isom}.
\end{proof}

We now focus on proving that the functors $\Phi^{s*}$ give equivalences of categories. We begin by defining the functor that will provide the inverse.

\begin{definition}
	Let $R$ be a ring with CL-FFRT, fix notation as in Setup \ref{setup-FFRT}, and let $s \geq 0$ be an integer. The functor $\cC^s: \cD^{(s)}\mmod \to \cD^{(0)}\mmod$ is given by $\cC^s := \Hom_R(F^{(s)}_* R , F^{(0)}_* R) \otimes_{\cD^{(s)}} (-)$. Note that this functor depends on the choices of $a$ and $b$ made in Setup \ref{setup-FFRT}.
\end{definition}

Just as before, this induces functors on higher levels, as indicated below. The proof of these statements is analogous to the proofs of Proposition \ref{prop-higher-Frob-pullback} and Proposition \ref{prop-Frob-pullback-compo}.

\begin{proposition} \label{prop-higher-Cartier-functor-comp}
	Let $R$ be a ring with {\upshape CL-FFRT}, fix notation as in Setup \ref{setup-FFRT}, and let $e, s, t \geq 0$ be integers.
\begin{enuroman}
\item If $M$ is a left $\cD^{(e+s)}$-module then $\cC^s(M)$ admits a natural $\cD^{(e)}$-module structure, and the induced functor $\cD^{(e+s)}\mmod \to \cD^{(e)}\mmod$ is given by
	$$\Hom_R(F^{(e+s)}_* R, F^{(e)}_* R) \otimes_{\cD^{(e+s)}} (-).$$
\item The functor $\cC^{s+t}: \cD^{(e+s+t)}\mmod \to \cD^{(e)}\mmod$ is naturally equivalent to the composition of functors $\cC^t \circ \cC^s: \cD^{(e+s+t)}\mmod \to \cD^{(e+t)}\mmod \to \cD^{(e)}\mmod$.
\end{enuroman}
\end{proposition}
%
%\begin{proposition} \label{prop-Cartier-functor-comp}
%	Let $R$ be a ring with {\upshape CL-FFRT}, fix notation as in Setup \ref{setup-FFRT}, and let $e, s, t \geq 0$ be nonnegative integers. The functor $\cC^{s+t}: \cD^{(e+s+t)}\mmod \to \cD^{(e)}\mmod$ is naturally equivalent to the composition of functors $\cC^t \circ \cC^s: \cD^{(e+s+t)}\mmod \to \cD^{(e+t)}\mmod \to \cD^{(e)}\mmod$.
%\end{proposition}
%

\begin{theorem} \label{thm-Frob-equivalence}
	Let $R$ be a ring with {\upshape CL-FFRT}, and fix notation as in Setup \ref{setup-FFRT}. For every integers $e, s \geq 0$ the functors $\Phi^{s*}$ and $\cC^s$ define a natural equivalence of categories
	$$\begin{tikzcd}
		\cD^{(e)}\mmod \arrow[r, "\Phi^{s*}", shift left] \arrow[r, "\sim" {yshift = 1pt}, shift left, swap] & \cD^{(e+s)}\mmod. \arrow[l, "\cC^s", shift left]
	\end{tikzcd}$$
\end{theorem}
\begin{proof}
	By virtue of Proposition \ref{prop-higher-Frob-pullback} and \ref{prop-higher-Cartier-functor-comp}, we see that the composition $\cC^s \circ \Phi^{s*}$ is given by 
	$$\cC^s \circ \Phi^{s*} = \Hom_R(F^{(e+s)}_* R, F^{(e)}_* R) \otimes_{\cD^{(e+s)}} \Hom_R(F^{(e)}_* R, F^{(e+s)}_* R) \otimes_{\cD^{(e)}} (-).$$
	Composition induces a natural $(\cD^{(e)}, \cD^{(e)})$-bimodule map 
	$$\Hom_R(F^{(e+s)}_* R, F^{(e)}_* R) \otimes_{\cD^{(e+s)}} \Hom_R(F^{(e)}_* R, F^{(e+s)}_* R) \longrightarrow \Hom_R(F^{(e)}_* R , F^{(e)}_* R) = \cD^{(e)},$$
	which we claim is an isomorphism. Note that, assuming the claim, we get $\cC^s \circ \Phi^{s*} \simeq \cD^{(e)} \otimes_{\cD^{(e)}} (-) \simeq \id$.
	
	To prove the claim we may localize and complete, and hence we may assume that $R$ is complete local and that all of the modules $F^{(i)}_* R$ have the same indecomposable summands. The claim then follows from Lemma \ref{lemma-compo-isom}. 
	
	Similarly, the composition $\Phi^{s*} \circ \cC^s$ is given by 
	$$\Phi^{s*} \circ \cC^s  = \Hom_R(F^{(e)}_* R, F^{(e+s)}_* R) \otimes_{\cD^{(e)}} \Hom_R(F^{(e+s)}_* R, F^{(e)}_* R) \otimes_{\cD^{(e+s)}} (-),$$
	composition induces a bimodule isomorphism 
	$$\Hom_R(F^{(e)}_* R, F^{(e+s)}_* R) \otimes_{\cD^{(e)}} \Hom_R(F^{(e+s)}_* R, F^{(e)}_* R) \longrightarrow \Hom_R(F^{(e+s)}_* R, F^{(e+s)}_* R) = \cD^{(e+s)},$$
	and hence $\Phi^{s*} \circ \cC^s \simeq \id$. 
\end{proof}

\begin{corollary}
	The functor $\Phi^{s*}: \cD^{(e)}\mmod \to \cD^{(e+s)} \mmod$ is naturally equivalent to
	$$\Phi^{s*} \simeq \Hom_{\cD^{(e)}} \big( \Hom_R(F^{(e+s)}_* R, F^{(e)}_* R), - \big)$$
	and, similarly, the functor $\cC^s: \cD^{(e+s)}\mmod \to \cD^{(e)}\mmod$ is naturally equivalent to
	$$\cC^s \simeq \Hom_{\cD^{(e+s)}} \big( \Hom_R(F^{(e)}_* R, F^{(e+s)}_* R), - \big).$$
\end{corollary}
\begin{proof}
  This is a standard fact from Morita theory \cite[Thm. 22.1]{Anderson-Fuller-92}.
\end{proof}

Next, we use Theorem \ref{thm-Frob-equivalence} to give an alternative characterization of the category of $\cD$-modules on a ring with CL-FFRT. In the regular case this is usually loosely known as ``Frobenius descent" (see \cite{AMBL}). Note that the remaining portion of this section will not be used in our applications to local cohomology.

\begin{definition}
	Let $R$ be a ring with CL-FFRT, and fix notation as in Setup \ref{setup-FFRT}. A $\Phi$-divided left $\cD^{(0)}$-module $(M_e, \alpha_e)_{e = 0}^\infty$ consists of the following data:
	\begin{enualph}
	\item A collection of left $\cD^{(0)}$-modules $(M_0, M_1, \dots)$.
	\item For every integer $e \geq 0$, an isomorphism $\alpha_e : \Phi^*(M_{e+1}) \xrightarrow{\sim} M_e$ of $\cD^{(0)}$-modules, where $\Phi^*(M_{e+1})$ is viewed as a $\cD^{(0)}$-modules via restriction of scalars from $\cD^{(1)}$. 
	\end{enualph}
	The collection of such objects is a category in a natural way: a morphism $\phi: (M_e, \alpha_e)_{e =0}^\infty \to (N_e, \beta_e)_{e = 0}^\infty$ is a collection of $\cD^{(0)}$-linear homomorphisms $\phi_e: M_e \to N_e$ such that for every $e \geq 0$ the diagram
	$$
	\begin{tikzcd}
		\Phi^*(M_{e+1}) \arrow[d, "\Phi^* \phi_{e+1}", swap] \arrow[r, "\alpha_e"] & M_e \arrow[d, "\phi_e"] \\
		\Phi^*(N_{e+1}) \arrow[r, "\beta_e"] & N_e
	\end{tikzcd}
	$$
	commutes. We denote this category by $\cD^{(0)} \m \Phi \m \divmod$. Note that this category depends on the choices of $a,b$ made in Setup \ref{setup-FFRT}.
\end{definition}

We will show that this category is equivalent to the category of left $\cD$-modules. We begin by explaining how to produce a $\Phi$-divided left $\cD^{(0)}$-module from a left $\cD$-module. Given a left $\cD$-module $M$ and an integer $e \geq 0$, we can view $M$ as a $\cD^{(e)}$-module via restriction of scalars, and thus apply to it the functor $\cC^e: \cD^{(e)}\mmod \to \cD^{(0)} \mmod$ to get a left $\cD^{(0)}$-module $M_e := \cC^e (M)$. Using Proposition \ref{prop-higher-Cartier-functor-comp} and Theorem \ref{thm-Frob-equivalence}, we get canonical isomorphisms
$$\Phi^* \circ \cC^e(M) \cong \Phi^* \circ \cC \circ \cC^{e-1}(M) \cong \cC^{e-1}(M),$$
i.e., isomorphisms $\alpha_e: \Phi^*(M_e) \xrightarrow{\sim} M_{e-1}$. We conclude that $(M_e, \alpha_e)$ is a $\Phi$-divided left $\cD^{(0)}$-module. It is also clear that the construction $[M \to (M_e, \alpha_e)_{e = 0}^\infty ]$ is functorial.

Conversely, suppose $(M_e, \alpha_e)_{e = 0}^\infty$ is a $\Phi$-divided left $\cD^{(0)}$-module, and let us show that $M_0$ then has a natural $\cD$-module structure. Note that given an integer $e \geq 0$ the composition $\alpha_1 \circ \Phi^*\alpha_2 \circ \cdots \circ \Phi^{(e-1)*} \alpha_e$ gives an isomorphism $\Phi^{e*}(M_e) \xrightarrow{\sim} M_0$. Since $\Phi^{e*}(M_e)$ is naturally a left $\cD^{(e)}$-module, we can equip $M_0$ with the unique left $\cD^{(e)}$-module structure that makes this isomorphism $\cD^{(e)}$-linear. 

We now need to check that these $\cD^{(e)}$-module structures are compatible as $e$ varies. For this, fix an $e \geq 0$, let $M_{0}^{(e)}$ be $M_0$ with the $\cD^{(e)}$-module structure induced from $\Phi^{e*}(M_e)$, and let $M_{0}^{(e+1)}$ be $M_0$ with the $\cD^{(e+1)}$-module structure induced from $\Phi^{(e+1)*}(M_{e+1})$. The compatibility can be rephrased as saying that the identity map $M_0^{(e+1)} \xrightarrow{=} M_0^{(e)}$ is $\cD^{(e)}$-linear. To see that this is the case, observe that because of the way the maps $\Phi^{e*}(M_e) \xrightarrow{\sim} M_0$ and $\Phi^{(e+1)*}(M_{e+1}) \xrightarrow{\sim} M_0$ were constructed, the diagram
$$
\begin{tikzcd}[column sep = normal]
	\Phi^{(e+1)*}(M_{e+1}) \arrow[d, swap, "\sim" {anchor = north, rotate = 90}] \arrow[d, "\Phi^{e*} \alpha_e", swap]  \arrow[r, "\sim" ] & M_0^{(e+1)} \arrow[d, "="{anchor = south, rotate = 90}] \\
	\Phi^{e*}(M_e) \arrow[r, "\sim"] &   M_0^{(e)} 
\end{tikzcd}
$$
commutes, and that $\Phi^{e*}(\alpha_e)$ is $\cD^{(e)}$-linear because $\Phi^{e*}$ is a functor $\cD^{(0)}\mmod \to \cD^{(e)}\mmod$.

\begin{lemma}
	Let $(M_e, \alpha_e)_{e = 0}^\infty$ and $(N_e, \beta_e)_{e = 0}^\infty$ be $\Phi$-divided left $\cD^{(0)}$-modules, and let $(\phi_e: M_e \to N_e)_{e = 0}^\infty$ be a morphism. After equipping $M_0$ and $N_0$ with the $\cD$-module structure defined above, the map $\phi_0: M_0 \to N_0$ is $\cD$-linear.
\end{lemma}
\begin{proof}
	We claim that for every $e \geq 0$ the diagram
$$
\begin{tikzcd}
	\Phi^{e*}(M_e) \arrow[d, "\Phi^{e*}\phi_e", swap] \arrow[r, "\sim"] & M_0 \arrow[d, "\phi_0"] \\
	\Phi^{e*}(N_e) \arrow[r, "\sim"] & N_0 
\end{tikzcd}
$$
commutes, i.e., that we have $\phi_0 \circ \alpha_0 \circ \Phi^* \alpha_1 \circ \cdots \Phi^{(e-1)*} \alpha_{e-1} = \beta_0 \circ \Phi^* \beta_1 \circ \cdots  \circ \Phi^{(e-1)*} \beta_{e-1} \circ \Phi^{e*} \phi_e$. We prove this by induction on $e \geq 1$.

Note that, because $(\phi_e)_{e = 0}^\infty$ is a morphism of $\Phi$-divided left $\cD^{(0)}$-modules, we get that $\phi_i \circ \alpha_i = \beta_i \circ \Phi^* \phi_{i +1}$ for every $i \geq 0$, and plugging in $i = 0$ we get the base case $e = 1$. Now assuming the claim holds for an integer $e \geq 1$, note that
\begin{align*}
	\phi_0 \circ \alpha_0 \circ \Phi^* \alpha_1 \circ \cdots \Phi^{e*} \alpha_{e} & = \beta_0 \circ \Phi^* \beta_1 \circ \cdots  \circ \Phi^{(e-1)*} \beta_{e-1} \circ \Phi^{e*} \phi_e \circ \Phi^{e*} \alpha_{e +1} \\
		& = \beta_0 \circ \Phi^* \beta_1 \circ \cdots  \circ \Phi^{(e-1)*} \beta_{e-1} \circ \Phi^{e*} \big( \phi_e \circ  \alpha_{e +1} \big) \\
		& = \beta_0 \circ \Phi^* \beta_1 \circ \cdots  \circ \Phi^{(e-1)*} \beta_{e-1} \circ \Phi^{e*} \big( \beta_e \circ \Phi^* \phi_{e + 1} \big) \\
		& = \beta_0 \circ \Phi^* \beta_1 \circ \cdots  \circ \Phi^{(e-1)*} \beta_{e-1} \circ \Phi^{e*} \beta_e \circ \Phi^{(e+1)*} \phi_{e+1}. \qedhere
\end{align*}
\end{proof}

\begin{theorem} \label{thm-Frob-desc}
	Let $R$ be a ring with {\upshape CL-FFRT}, and fix notation as in Setup \ref{setup-FFRT}. Then the functors given above induce a natural equivalence of categories
	$$\begin{tikzcd}
		\cD \mmod \arrow[r,shift left] \arrow[r, "\sim" {yshift = 1pt}, shift left, swap] & \cD^{(0)} \m \Phi \m \divmod. \arrow[l, shift left]
	\end{tikzcd}$$
\end{theorem}
\begin{proof}
	We check that both compositions return the identity functor up to natural isomorphism. If $M$ is a left $\cD$-module, the composition $\cD\mmod \to \cD^{(0)} \m \Phi \m \divmod \to \cD\mmod$ sends $[M \mapsto (M, \cC (M), \cC^2(M) , \dots) \mapsto M]$. By Theorem \ref{thm-Frob-equivalence}, for every $e \geq 0$ we have an isomorphism $\Phi^{e*} \cC^e(M) \cong M$ of $\cD^{(e)}$-modules, and we conclude from this that the output has the same $\cD$-module structure as the input. If $\phi: M \to N$ is a morphism of left $\cD$-modules, the composition $\cD\mmod \to \cD^{(0)} \m \Phi \m \divmod \to \cD\mmod$ sends $\phi$ to $[\phi \mapsto (\phi, \cC \phi, \cC^2 \phi, \dots) \mapsto \phi]$. We conclude that the composition $\cD\mmod \to \cD^{(0)} \m \Phi \m \divmod \to \cD\mmod$ is the identity functor.

	We now consider the composition $\cD^{(0)} \m \Phi \m \divmod \to \cD\mmod \to \cD^{(0)} \m \Phi \m \divmod$. A $\Phi$-divided left $\cD^{(0)}$-module $(M_e, \alpha_e)_{e = 0}^\infty$ is sent by this composition to $[(M_e, \alpha_e)_{e = 0}^\infty \mapsto M_0 \mapsto (\cC^e(M_0), \beta_e)_{e = 0}^\infty]$; we claim that $(M_e, \alpha_e)_{e = 0}^\infty$ is naturally isomorphic to $(\cC^e(M_0), \beta_e)_{e = 0}^\infty$ as $\Phi$-divided left $\cD^{(0)}$-modules.

	For this, recall that for every $e \geq 0$ the $\cD^{(e)}$-module structure of $M_0$ is given through the isomorphism $\Phi^{e*}(M_e) \xrightarrow{\sim} M_0$ coming from $(\alpha_e)_{e = 0}^\infty$, and we let $\gamma_e$ denote the composition 
	$$\gamma_e: M_e \xrightarrow{\sim} \cC^e \circ \Phi^{e*} (M_e) \xrightarrow{\cC^e \alpha_e} \cC^e (M_0).$$
	We need to check that $\gamma_e$ is a morphism of $\Phi$-divided left $\cD^{(0)}$-modules. For this, consider the diagram
	$$
	\begin{tikzcd}
		\Phi^*(M_{e+1}) \arrow[d, equals] \arrow[r, "\sim"] \arrow[rr, bend left = 15, "\Phi^* \gamma_{e + 1}"] \arrow[dd, bend right = 60, "\alpha_e", swap] & \Phi^* \circ \cC^{e+1} \circ \Phi^{(e+1)*} (M_{e + 1}) \arrow[r, "\sim"] \arrow[d, "\sim" {anchor = south, rotate = 90}] & \Phi^* \circ \cC^{e+1} (M_0) \arrow[d, "\sim" {anchor = south, rotate = 90}] \arrow[dd, bend left = 60, "\beta_e"]\\
		\Phi^*(M_{e+1}) \arrow[r, "\sim"]  \arrow[d, "\sim" {anchor = south, rotate = 90}]& \cC^e \circ \Phi^{(e+1)*} (M_{e+1}) \arrow[r, "\sim"] \arrow[d, "\sim" {anchor = south, rotate = 90}]& \cC^e(M_0) \arrow[d, equals]\\
		M_e \arrow[r, "\sim"] \arrow[rr, bend right = 15, "\gamma_e", swap] & \cC^e \circ \Phi^{e*} (M_e) \arrow[r, "\sim"] & \cC^e(M_0).
	\end{tikzcd}
	$$
	We claim that this diagram is commutative; note that this will prove that $\gamma_e$ is a morphism of $\Phi$-divided $\cD^{(0)}$-modules. The top right square is a commutative diagram associated to the isomorphism of functors $\id \simeq \Phi^* \circ \cC$, the bottom left square is a commutative diagram associated to the isomorphism of functors $\id \simeq \cC^e \circ \Phi^{e*}$, and the bottom right square is obtained by applying the functor $\cC^e$ to a diagram we know to be commutative. It hence suffices to check that the top left square is commutative.

	For compactness of notation, let $F^{(i)} := F^{(i)}_* R$ for every $i \geq 0$. We observe that the following diagram of bimodules, where all arrows are given by composition, is commutative:
	$$
	\begin{tikzcd}
		\Hom_R(F^{(0)}, F^{(1)})  & \Hom_R(F^{(0)}, F^{(1)}) \otimes_{\cD^{(0)}} \Hom_R(F^{(e+1)}, F^{(0)}) \otimes_{\cD^{(e+1)}} \Hom_R(F^{(0)}, F^{(e+1)}) \arrow[l, "\sim", swap] \arrow[d, "\sim" {anchor = north, rotate = 90}]\\
		\Hom_R(F^{(0)}, F^{(1)}) \arrow[u, equals] & \Hom_R(F^{(e+1)}, F^{(1)}) \otimes_{\cD^{(e+1)}} \Hom_R(F^{(0)}, F^{(1)}). \arrow[l, "\sim", swap]
	\end{tikzcd}
	$$
	By using the descriptions of $\Phi^{i*}$ and $\cC^{i}$ given in Propositions \ref{prop-higher-Frob-pullback} and \ref{prop-higher-Cartier-functor-comp}, we see that the square in question is obtained by applying $(-) \otimes_{\cD^{(0)}} M_{e+1}$ to a commutative diagram, and is therefore commutative. This completes the proof of the claim.

	Finally, suppose that $\phi = (\phi_e)_{e = 0}^\infty : (M_e, \alpha_e)_{e = 0}^\infty \to (N_e, \beta_e)_{e = 0}^\infty$ is a morphism of $\Phi$-divided left $\cD^{(0)}$-modules. For every $e \geq 0$ we get a commutative diagram
	$$
	\begin{tikzcd}
		M_e \arrow[r, "\sim"] \arrow[d, "\phi_e", swap] & \cC^e \circ \Phi^{e*} (M_e) \arrow[r, "\sim"] \arrow[d, "\cC^e \Phi^{e*} \phi_e"] & \cC^e(M_0) \arrow[d, "\cC^e \phi_0"] \\
		N_e \arrow[r, "\sim"] & \cC^e \circ \Phi^{e*} (N_e) \arrow[r, "\sim"] & \cC^e(N_0).
	\end{tikzcd}
	$$
	This entails that in $\cD^{(0)} \m \Phi \m \divmod$ the diagram
	$$
	\begin{tikzcd}
		(M_e)_{e = 0}^\infty \arrow[r, "\sim"] \arrow[d, "(\phi_e)"] & (\cC^e(M_0))_{e = 0}^\infty \arrow[d, "(\cC^e \phi_0)"] \\
		(N_e)_{e = 0}^\infty \arrow[r, "\sim"] & (\cC^e(N_0))_{e = 0}^\infty
	\end{tikzcd}
	$$
	also commutes. Therefore, the composition $\cD^{(0)} \m \Phi \m \divmod \to \cD\mmod \to \cD^{(0)} \m \Phi \m \divmod$ is isomorphic to the identity.
\end{proof}

\begin{corollary}
	The category $\cD^{(0)} \m \Phi \m \divmod$ is, up to natural equivalence, independent of the choices of $a, b$ made in Setup \ref{setup-FFRT}.
\end{corollary}
\begin{proof}
	The category $\cD\mmod$ is independent of the choices of $a,b$.
\end{proof}

\begin{corollary}
	The functor $\Phi^*$ defines a self-equivalence of the category $\cD \mmod$, with inverse given by $\cC$.
\end{corollary}
\begin{proof}
	The statement is very clear when interpreted in the category $\cD^{(0)} \m \Phi \m \divmod$.
\end{proof}

\section{A formalism of $F$-modules} \label{scn-F-mod}

Let $R$ be a ring with CL-FFRT, and fix notation as in Setup \ref{setup-FFRT}. Note that, given $R$-linear maps $\alpha: F^{(0)}_* R \to F^{(i)}_* R$ and $\beta: F^{(0)}_*R \to F^{(j)}_* R$, their composition defines an $R$-linear map $\beta \circ \alpha: F^{(0)}_* R \to F^{(i + j)}_* R$. This means that the object
$$\cD^{(0)} \la \Phi \ra := \bigoplus_{i = 0}^\infty \Hom_R(F^{(0)}_* R , F^{(i)}_* R)$$
can be given a ring structure, where multiplication is given by composition. Note that this ring is not commutative nor Noetherian in general, and that it depends on the choice of $a$ and $b$ coming from Setup \ref{setup-FFRT}.

The formation of $\cD^{(0)} \la \Phi \ra$ is compatible with localization (as long as one chooses the same $a$ and $b$ for the localization; see Remark \ref{rmk-FFRT-couple-localization}). In particular, the quasicoherent sheaf on $\Spec(R)$ associated to $\cD^{(0)} \la \Phi \ra$ is naturally a sheaf of rings. 

\begin{example}
	Suppose $R$ is regular, and that $a = 0$, $b = 1$, so that $F^{(e)} = F^e$ for every $e \geq 0$. Note that every $R$-linear map $R \to F^{i}_* R$ can be written as $r F^i$ for some $r \in R$; this means that $\Hom_R(R, F^i_*R ) = R F^i$. We then get $\cD^{(0)} \la \Phi \ra = \oplus_i R F^i$, i.e., we get $\cD^{(0)} \la \Phi \ra = R \la F \ra$, where $R \la F \ra$ is in the sense of Emerton-Kisin \cite{EK04} (see also \cite[Ch. 2]{Blickle-thesis}). Our constructions therefore naturally extend the notions of $F$-modules of Emerton and Kisin. 
\end{example}

We will be interested in the study of left $\cD^{(0)} \la \Phi \ra$-modules. Examples of such modules include $R$, its localizations, and its local cohomology modules (see Section \ref{scn-app-to-local-coh}). 

We recall that given a left $\cD^{(0)}$-module $M$ we have
$$\Phi^*(M) := \Hom_R(F^{(0)}_* R, F^{(1)}_* R) \otimes_{\cD^{(0)}} M,$$
which is naturally a left $\cD^{(1)}$-module, hence a left $\cD^{(0)}$-module by restriction of scalars (see Definition \ref{def-FFRT-F-pullback}).  

For the following proposition we regard $\Hom_R(F^{(0)}_*R , F^{(1)}_* R)$ as a $(\cD^{(0)}, \cD^{(0)})$-bimodule, and we use the notation from Subsection \ref{subscn-bimodules}.

\begin{proposition} \label{prop-tensor-compo-isom}
	Let $R$ be a ring with {\upshape CL-FFRT}, and fix notation as in Setup \ref{setup-FFRT}. For every $e \geq 0$, the natural map 
	$$\Hom_R(F^{(0)}_* R, F^{(1)}_* R)^{\otimes e} \to \Hom_R(F^{(0)}_*R, F^{(e)}_* R)$$
	given by $[\phi_1 \otimes \cdots \otimes \phi_e \mapsto \phi_1 \circ \cdots \circ \phi_e]$ is an isomorphism. In particular, $\cD^{(0)} \la \Phi \ra$ is the tensor algebra for the $(\cD^{(0)}, \cD^{(0)})$-bimodule $\Hom_R(F^{(0)}_* R, F^{(1)}_* R)$. 
\end{proposition}
\begin{proof}
	Follows from an iterated use of Lemma \ref{lemma-compo-isom-2}. 
\end{proof}

\begin{corollary} \label{cor-phi-mod-equiv}
	Giving a left $\cD^{(0)} \la \Phi \ra$-module $M$ is equivalent to giving a left $\cD^{(0)}$-module $M$ together with a $\cD^{(0)}$-linear map
	$$\psi_M: \Phi^*(M) \to M.$$
	Moreover, if $M$ and $N$ are left $\cD^{(0)} \la \Phi \ra$-modules and $f: M \to N$ is a $\cD^{(0)}$-linear map, then $f$ is $\cD^{(0)} \la \Phi \ra$-linear if and only if the following diagram is commutative:
	$$\begin{tikzcd}
		\Phi^*(M) \arrow[d, "\psi_M"] \arrow[r, "\Phi^*(f)"] & \Phi^*(N) \arrow[d, "\psi_N"] \\
		M \arrow[r, "f"]& N.
	\end{tikzcd}$$
\end{corollary}

\begin{definition}
	Let $R$ be a ring with CL-FFRT, and fix notation as in Setup \ref{setup-FFRT}. Let $M$ be a left $\cD^{(0)} \la \Phi \ra$-module. We say that $M$ is unit whenever the natural map $\psi_M : \Phi^*(M) \to M$ is an isomorphism. We denote by $\cD^{(0)} \la \Phi \ra \mumod$ (resp. $\cD^{(0)} \la \Phi \ra \mfgumod)$ the full subcategory of $\cD^{(0)} \la \Phi \ra \mmod$ spanned by the modules that are unit (resp. finitely generated and unit).
\end{definition}

\begin{proposition} \label{prop-unit-kce-closed}
	The subcategory $\cD^{(0)}\la \Phi \ra \mumod \sq \cD^{(0)} \la \Phi \ra \mmod$ is closed under kernels, cokernels, and extensions. In particular, $\cD^{(0)} \la \Phi \ra \mumod$ is an abelian category.
\end{proposition}
\begin{proof}
	Suppose $0 \to N \to M \to Q \to 0$ is an exact sequence in $\cD^{(0)} \la \Phi \ra \mmod$. By Theorem \ref{thm-Frob-equivalence}, $\Phi^*$ is an equivalence of categories, and therefore exact. By Corollary \ref{cor-phi-mod-equiv} we obtain a commutative diagram
	$$\begin{tikzcd}
		0 \arrow[r] & \Phi^*(N) \arrow[d, "\psi_N"] \arrow[r] & \Phi^*(M) \arrow[d, "\psi_M"] \arrow[r] & \Phi^*(Q) \arrow[d, "\psi_Q"] \arrow[r] & 0 \\
		 0 \arrow[r] & N \arrow[r] & M \arrow[r] & Q \arrow[r] & 0
	\end{tikzcd}$$
	with exact rows. By the five-lemma, if two of the vertical arrows are isomorphisms then so is the third.
\end{proof}

Let $M$ be a unit left $\cD^{(0)} \la \Phi \ra$-module and $N \sq M$ be a $\cD^{(0)}$-submodule. By Theorem \ref{thm-Frob-equivalence}, $\Phi^*$ is an equivalence of categories, and in particular exact. It follows that $\Phi^*(N) \to \Phi^*(M)$ is injective, and the composition $\Phi^*(N) \to \Phi^*(M) \xrightarrow{\sim} M$ identifies $\Phi^*(N)$ with the following submodule of $M$:
$$\Phi^*(N) \cong \Hom_R(F^{(0)}_* R, F^{(1)}_* R) \cdot N \sq M$$
(recall that $\Hom_R(F^{(0)}_* R, F^{(1)}_* R)$ is a submodule of $\cD^{(0)} \la \Phi \ra$, and therefore acts on elements of $M$). From now on, by an abuse of notation we write $\Phi^*(N)$ for this submodule of $M$; this makes the isomorphism above an equality. We similarly identify $\Phi^{s*}(N)$ with $\Hom_R(F^{(0)}_* R, F^{(s)}_* R) \cdot N$. It follows from Proposition \ref{prop-tensor-compo-isom} that, with this notation, we have $\Phi^{s*}(\Phi^{t*}(N)) = \Phi^{(s+t)*}(N)$ for all integers $s, t \geq 0$.

\begin{definition}
	Let $R$ be a ring with CL-FFRT. Fix the notation as in Setup \ref{setup-FFRT}, and let $M$ be a unit left $\cD^{(0)} \la \Phi \ra$-module. A finitely generated $\cD^{(0)}$-submodule $M_0 \sq M$ is called a root of $M$ if the following two conditions are satisfied:
	\begin{enuroman}
	\item We have $M_0 \sq \Phi^*(M_0)$ (and thus $M_0 \sq \Phi^*(M_0) \sq \Phi^{2*}(M_0) \sq \cdots )$. 
	\item We have $M = \bigcup_{e = 0}^\infty \Phi^{e*}(M_0)$. 
	\end{enuroman}
\end{definition}

Note that if $M_0 \sq M$ is a $\cD^{(0)}$-submodule satisfying (i) then Proposition \ref{prop-tensor-compo-isom} tell us that 
$$\bigcup_{e = 0}^\infty \Phi^{e*}(M_0) = \bigcup_{e = 0}^\infty \Hom_R(F^{(0)}_* R, F^{(e)}_* R) \cdot M_0 = \cD^{(0)} \la \Phi \ra \cdot M_0,$$
and therefore (ii) is just asking for $M_0$ to generate $M$ as a $\cD^{(0)} \la \Phi \ra$-module.

In the regular case, the following result and its proof can be found in \cite[Prop. 2.11]{Blickle-thesis}.

\begin{proposition} \label{prop-fg-iff-root}
	Let $R$ be a ring with {\upshape CL-FFRT}, fix notation as in Setup \ref{setup-FFRT}, and let $M$ be a unit $\cD^{(0)} \la \Phi \ra$-module. Then $M$ admits a root if and only if $M$ is finitely generated over $\cD^{(0)} \la \Phi \ra$. 
\end{proposition}
\begin{proof}
	If $M$ admits a root $M_0$ then, as discussed above, the generators of $M_0$ generate $M$ as a $\cD^{(0)} \la \Phi \ra$-module. Conversely, suppose $M$ is finitely generated over $\cD^{(0)} \la \Phi \ra$. Then there exists some finitely generated $\cD^{(0)}$-submodule $N_0 \sq M$ such that $M = \cD^{(0)} \la \Phi \ra \cdot N_0$ (take $N_0$ to be the $\cD^{(0)}$-submodule generated by finitely many $\cD^{(0)} \la \Phi \ra$-generators of $M$).

	Since $M$ is unit we have $\Phi^*(M) = M$, and hence
	$$N_0 \sq \Phi^*(\cD^{(0)} \la \Phi \ra \cdot N_0) = \sum_{e = 1}^\infty \Phi^{e*}(N_0).$$
	Since $N_0$ is finitely generated over $\cD^{(0)}$ there exists some integer $K > 0$ such that $N_0 \sq \sum_{e = 1}^K \Phi^{e*}(N_0)$. Take $M_0 = \sum_{e = 0}^{K-1} \Phi^{e*}(N)$; in other words, $M_0$ is the image of the natural map
	$$\bigoplus_{e = 0}^{K-1} \Hom_R(F^{(0)}_* R, F^{(1)}_* R) \otimes_{\cD^{(0)}} N_0 \to M.$$
	Since the left hand side is finitely generated over $R$, we conclude that $M_0$ is finitely generated over $R$, and thus finitely generated over $\cD^{(0)}$. We claim that $M_0$ is a root for $M$. 

	To see this, first note that $M_0$ contains $N_0$, and therefore $M = \cD^{(0)} \la \Phi \ra \cdot M_0$, i.e., condition (ii) is satisfied. For condition (i) simply observe that
	\[\Phi^*( M_0 ) = \sum_{e = 1}^K \Phi^{e*}(N_0) \supseteq N_0 + \sum_{e = 1}^{K-1} \Phi^{e*}(N_0) = M_0. \qedhere \]
\end{proof}

In the theory of $F$-modules for regular rings, one crucially uses the fact that Frobenius pullback (i.e., the Peskine-Szpiro functor) is exact and thus commutes with finite intersections. In the next lemma we observe that the same is true for $\Phi^*$; recall that $\Phi^*$ is exact because it is an equivalence of categories (Theorem \ref{thm-Frob-equivalence}).

\begin{remark} \label{rmk-Phi*-commutes-intersection}
Let $R$ be a ring with {\upshape CL-FFRT}, and fix notation as in Setup \ref{setup-FFRT}. Let $C$ be a left $\cD^{(0)}$-module, let $A, B \sq C$ be $\cD^{(0)}$-submodules, and let $e \geq 0$ be an integer. Applying $\Phi^{e*}$ to the exact sequence $0 \to A \cap B \to C \to (C / A) \oplus (C / B)$, we see that $\Phi^{e*}(A \cap B)$ is identified with the kernel of $\Phi^{e*}(C) \to (\Phi^{e*}(C) / \Phi^{e*}(A)) \oplus (\Phi^{e*}(C) / \Phi^{e*}(B))$, which is in turn identified with $\Phi^{e*}(A) \cap \Phi^{e*}(B)$. Therefore, after identifying $\Phi^{e*}(A), \Phi^{e*}(B)$, and $\Phi^{e*}(A \cap B)$ with their images in $\Phi^{e*}(C)$, we have
$$\Phi^{e*}(A) \cap \Phi^{e*}(B) = \Phi^{e*}(A \cap B).$$
\end{remark}

%\begin{lemma} \label{lemma-Phi*-commutes-intersection}
%	Let $R$ be a ring with {\upshape CL-FFRT}, and fix notation as in Setup \ref{setup-FFRT}. Let $C$ be a left $\cD^{(0)}$-module, let $A, B \sq C$ be $\cD^{(0)}$-submodules, and let $e \geq 0$ be an integer. Identify $\Phi^{e*}(A), \Phi^{e*}(B)$, and $\Phi^{e*}(A \cap B)$ with their images in $\Phi^{e*}(C)$. We then have
%	%
%	$$\Phi^{e*}(A) \cap \Phi^{e*}(B) = \Phi^{e*}(A \cap B).$$
%\end{lemma}
%%
%\begin{proof}
%	Applying $\Phi^{e*}$ to the exact sequence $0 \to A \cap B \to C \to (C / A) \oplus (C / B)$, we see that $\Phi^{e*}(A \cap B)$ is identified with the kernel of $\Phi^{e*}(C) \to (\Phi^{e*}(C) / \Phi^{e*}(A)) \oplus (\Phi^{e*}(C) / \Phi^{e*}(B))$, which is $\Phi^{e*}(A) \cap \Phi^{e*}(B)$.
%\end{proof}

The analogue of the following proposition in the regular case is due to Lyubeznik \cite[Cor 2.6]{Lyubeznik-Fmod} (see also \cite[Prop. 2.15]{Blickle-thesis}).
\begin{proposition} \label{prop-unit-sub-root-sub-correspondence}
	Let $R$ be a ring with {\upshape CL-FFRT}, and fix notation as in Setup \ref{setup-FFRT}. Let $M$ be a finitely generated unit left $\cD^{(0)} \la \Phi \ra$-module, and fix a root $M_0 \sq M$. Then:
	\begin{enuroman}
	\item If $N \sq M$ is a unit $\cD^{(0)} \la \Phi \ra$-submodule of $M$ then $N_0 := N \cap M_0$ is a root of $N$, and $\Phi^*(N_0) \cap M_0 = N_0$. 
	\item If $N_0 \sq M_0$ is a $\cD^{(0)}$-submodule with the property $\Phi^*(N_0) \cap M_0 = N_0$ then $N := \cD^{(0)} \la \Phi \ra \cdot N_0$ is a unit $\cD^{(0)} \la \Phi \ra$-submodule of $M$.
	\item The assignments $[N \to N \cap M_0]$ and $[\cD^{(0)} \la \Phi \ra \cdot N_0 \leftarrow N_0]$ define a one-to-one inclusion preserving correspondence
	$$
	\begin{tikzcd}
	\begin{Bmatrix}
		\text{ Unit $ \cD^{(0)} \la \Phi \ra $-} \\
		\text{ submodules of $M$ }
	\end{Bmatrix} \arrow[r, leftrightarrow, "\sim"] &
	\begin{Bmatrix}
		\text{ $\cD^{(0)}$-submodules $N_0 \sq M_0$ } \\
		\text{ with $\Phi^*(N_0) \cap M_0 = N_0$ }
	\end{Bmatrix}
	\end{tikzcd}
	$$
	\end{enuroman}
\end{proposition}
\begin{proof}
	We begin with (i). First observe that $N_0$ is a $\cD^{(0)}$-submodule of the finitely generated $\cD^{(0)}$-module $M_0$, and since $\cD^{(0)}$ is Noetherian (Proposition \ref{prop-De-noetherian}) we conclude that $N_0$ is finitely generated over $\cD^{(0)}$. Next, we use Remark \ref{rmk-Phi*-commutes-intersection}, the fact that $\Phi^*(N) = N$, and the fact that $M_0 \sq \Phi^*(M_0)$, to conclude that $\Phi^*(N \cap M_0) \cap M_0 = \Phi^*(N) \cap \Phi^*(M_0) \cap M_0 = N \cap M_0$, i.e., $\Phi^*(N_0) \cap M_0 = N_0$. In particular, this implies that $N_0 \sq \Phi^*(N_0)$. Finally, using Remark \ref{rmk-Phi*-commutes-intersection} once more, we get
	\begin{align*}
		\bigcup_{e = 0}^\infty \Phi^{e*}(N_0) & = \bigcup_{e = 0}^\infty \Phi^{e*}(N) \cap \Phi^{e*}(M_0) = \bigcup_{e = 0}^\infty N \cap \Phi^{e*}(M_0) \\
			& = N \cap \bigcup_{e = 0}^\infty \Phi^{e*}(M_0) = N \cap M = N,
	\end{align*}
	which concludes the proof of (i).
	
	For (ii), note that the assumption $\Phi^*(N_0) \cap M_0 = N_0$ immediately implies that $N_0 \sq \Phi^*(N_0)$, and hence $N$ can be expressed as an increasing union $N = \bigcup_{e = 0}^\infty \Phi^{e*}(N_0)$. In particular, $\Phi^*(N) = N$, which means that $N$ is unit (recall that the natural map $\Phi^*(N) \to \Phi^*(M) = M$ is always injective because $\Phi^*$ is exact).
	
	For (iii), we check that the assignments given are inverses of each other. First suppose that $N_0 \sq M_0$ is an $\cD^{(0)}$-submodule with $\Phi^*(N_0) \cap M_0 = N_0$. We claim that $\Phi^{e*}(N_0) \cap M_0 = N_0$ for every integer $e \geq 0$, the case $e = 0$ being trivial. If we assume that $\Phi^{e*}(N_0) \cap M_0 = N_0$, by applying $\Phi^*$ to both sides, and using Remark \ref{rmk-Phi*-commutes-intersection}, we get $\Phi^{(e+1)*}(N_0) \cap \Phi^*(M_0) = \Phi^*(N_0)$. Now intersecting with $M_0$, we get $\Phi^{(e+1)*}(N_0) \cap M_0 = \Phi^*(N_0) \cap M_0 = N_0$. The claim is thus true by induction.
	
	Using the claim, we get
	$$M_0 \cap \bigcup_{e =0}^\infty \Phi^{e*}(N_0) = \bigcup_{e = 0}^\infty \Phi^{e*}(N_0) \cap M_0 = N_0.$$
	Finally, suppose that $N \sq M$ is a unit $\cD^{(0)} \la \Phi \ra$-submodule of $M$. By part (i), we know that $N \cap M_0$ is a root for $N$, and therefore $\cD^{(0)} \la \Phi \ra \cdot (N \cap M_0) = N$.
\end{proof} 
\begin{corollary} \label{cor-unit-sub-fg}
	Suppose $M$ is a finitely generated unit left $\cD^{(0)} \la \Phi \ra$-module. Then every unit $\cD^{(0)} \la \Phi \ra$-submodule of $M$ is finitely generated.
\end{corollary}
\begin{proof}
	By Proposition \ref{prop-fg-iff-root} we know that there is a root $M_0$ for $M$. If $N \sq M$ is a unit $\cD^{(0)} \la \Phi \ra$-submodule then, by Proposition \ref{prop-unit-sub-root-sub-correspondence} we know that $N_0 := N \cap M_0$ is a root for $N$. Since $N$ has a root, it is finitely generated over $\cD^{(0)} \la \Phi \ra$ by Proposition \ref{prop-fg-iff-root}.
\end{proof}

\begin{corollary} \label{cor-fgu-abelian}
	The subcategory $\cD^{(0)} \la \Phi \ra \mfgumod \sq \cD^{(0)} \la \Phi \ra \mmod$ is closed under kernels, cokernels, and extensions. In particular, $\cD^{(0)} \la \Phi \ra \mfgumod$ is an abelian category. It is furthermore Noetherian, i.e., all ascending chains in $\cD^{(0)} \la \Phi \ra \mfgumod$ stabilize. 
\end{corollary}
\begin{proof}
	Let $0 \to N \to M \to Q \to 0$ be an exact sequence in $\cD^{(0)} \la \Phi \ra \mmod$; we want to show that if two of $N, M, Q$ are finitely generated unit then so is the third. First recall that whenever two of them are unit then so is the third (Proposition \ref{prop-unit-kce-closed}). Then recall that if $M$ is finitely generated then so is $Q$, and that whenever $N$ and $Q$ are both finitely generated then so is $M$. It remains to show that whenever $M$ and $Q$ are finitely generated unit then so is $N$. But this follows from Corollary \ref{cor-unit-sub-fg}.

	Let us now show that $\cD^{(0)} \la \Phi \ra \mfgumod$ is Noetherian. Fix a finitely generated unit left $\cD^{(0)} \la \Phi \ra$-module $M$ with root $M_0$. By Proposition \ref{prop-unit-sub-root-sub-correspondence}, any ascending chain $N^1 \sq N^2 \sq \cdots $ of unit $\cD^{(0)} \la \Phi \ra$-submodules of $M$ induces an ascending chain $N^1_0 \sq N^2_0 \sq \cdots$ of $\cD^{(0)}$-submodules of $M_0$ (where $N^i_0 := N^i \cap M_0$). Since $\cD^{(0)}$ is Noetherian (Proposition \ref{prop-De-noetherian}) and $M_0$ is finitely generated, the chain of $N^i_0$ must stabilize and, by Proposition \ref{prop-unit-sub-root-sub-correspondence}, so does the chain of $N^i$. 
\end{proof}

\begin{remark}
	In the regular case Lyubeznik also showed that finitely generated unit $R \la F \ra$-modules satisfy the descending chain condition on unit submodules. We do not know whether the analogous statement is true in this context.
\end{remark}

\begin{remark}
	If $\cA$ is an abelian category, a full subcategory $\cC \sq \cA$ is called wide whenever $\cC$ is closed under kernels, cokernels, and extensions \cite{Hovey-wide}; these are also sometimes called coherent \cite{Takahashi-classifying}. In this language, Proposition \ref{prop-unit-kce-closed} and Corollary \ref{cor-fgu-abelian} tell us that $\cD^{(0)} \la \Phi \ra \mumod$ and $\cD^{(0)} \la \Phi \ra \mfgumod$ are wide subcategories of $\cD^{(0)} \la \Phi \ra \mmod$.   
\end{remark}

\section{Applications to local cohomology} \label{scn-app-to-local-coh}

Let $R$ be a ring with CL-FFRT, and fix notation as in Setup \ref{setup-FFRT}. We start this section by showing that local cohomology modules $H^i_I(R)$ and, more generally, iterated local cohomology modules $H^{i_1}_{I_1} \circ \cdots \circ H^{i_r}_{I_r}(R)$ have natural left $\cD^{(0)} \la \Phi \ra$-module structures. We then show that, as such, they are finitely generated and unit, and from there one can use Lyubeznik's strategy to show that they must have a finite number of associated primes. 

We begin with the simple observation that $R$ is a left $\cD^{(0)} \la \Phi \ra$-module by declaring that, for every integer $e \geq 0$, every $\alpha \in \Hom_R(F^{(0)}_* R, F^{(e)}_* R) \sq \cD^{(0)} \la \Phi \ra$, and every $x \in R$ we have $\alpha \cdot x = \alpha(x)$.

\begin{proposition} \label{prop-R-is-fgu}
	As a left $\cD^{(0)} \la \Phi \ra$-module, $R$ is finitely generated unit.
\end{proposition}
\begin{proof}
	The finite generation is clear. For unit, we need to show that the natural map
	$$\Hom_R(F^{(0)}_* R, F^{(1)}_* R) \otimes_{\cD^{(0)}} F^{(0)}_* R \to F^{(1)}_* R$$
	given by $[\phi \otimes x \mapsto \phi(x)]$ is an isomorphism. This can be checked after localization and completion, and hence we may assume that $R$ is local and that all the $R$-modules $F^{(e)}_* R$ have the same summands. The claim then follows from Lemma \ref{lemma-compo-isom} by setting $M_1 = R, M_2 = F^{(0)}_* R$ and $M_3 = F^{(1)}_* R$.
\end{proof}

We now discuss a couple of results on the behavior of $\cD^{(0)} \la \Phi \ra$-modules under localization. We remark that, if $M$ is a left $\cD^{(0)} \la \Phi \ra$-module and $W \sq R$ is a multiplicative subset, we will still consider $W^{-1} M$ as a $\cD^{(0)} \la \Phi \ra$-module on $R$, and not on $W^{-1} R$. In particular, the preservation of finiteness properties under localization is not obvious.

\begin{proposition} \label{prop-localization-fg}
	Let $R$ be a ring with {\upshape CL-FFRT}, and fix notation as in Setup \ref{setup-FFRT}. Let $M$ be a left $\cD^{(0)} \la \Phi \ra$-module and $W \sq R$ be a multiplicative subset. Then there is a unique left $\cD^{(0)} \la \Phi \ra$-module structure on $W^{-1} M$ for which the localization map $M \to W^{-1} M$ is $\cD^{(0)} \la \Phi \ra$-linear. Moreover, if $M$ is unit then so is $W^{-1} M$. 
\end{proposition}
\begin{proof}
	The $\cD^{(0)} \la \Phi \ra$-module structure on $M$ is gives a map $\cD^{(0)} \la \Phi \ra \otimes_R M \to M$, and applying the functor $(-) \otimes_R W^{-1} R$ to this map gives the structure map $\cD^{(0)} \la \Phi \ra \otimes_R W^{-1} M \to W^{-1} M$. 

	This action can be described more explicitly: given $u \in M$, $w \in W$ and $\alpha \in \Hom_R(F^{(0)}_* R, F^{(e)}_* R) = \Hom_R(F^a_* R, F^{a + eb}_* R)$ we get
	$$\alpha \cdot \frac{u}{w} = \alpha \cdot \frac{u w^{p^a - 1}}{w^{p^a}} = \frac{\alpha \cdot u w^{p^a - 1}}{w^{p^{a + eb}}}.$$
	From this computation the uniqueness of the structure is clear. 

	Finally, recall that $M$ is unit precisely when the natural map $\Hom_R(F^{(0)}_* R, F^{(1)}_* R) \otimes_{\cD^{(0)}} M \to M$ is an isomorphism. Since the corresponding map for $W^{-1} M$ is the localization of this isomorphism, it is an isomorphism as well. 
\end{proof}

\begin{proposition} \label{prop-localization-single-elt-unit}
	Let $R$ be a ring with {\upshape CL-FFRT}, and fix notation as in Setup \ref{setup-FFRT}. Let $M$ be a finitely generated left $\cD^{(0)} \la \Phi \ra$-module and let $f \in R$ be an element. Then the $\cD^{(0)} \la \Phi \ra$-module $M[1/f]$ is also finitely generated. 
\end{proposition}
\begin{proof}
	Recall that being finitely generated is equivalent to the existence of a root (Proposition \ref{prop-fg-iff-root}). Fix a root $M_0$ for $M$, and we claim that $F^{(0)}(f^{-1})M_0 = f^{-p^a} M_0 \sq M[1/f]$ is a root for $M[1/f]$, where $a$ is as in Setup \ref{setup-FFRT}. First of all, we have 
	$$\Phi^*( F^{(0)}(f^{-1}) M_0) = F^{(1)}(f^{-1}) \Phi^*(M_0) \supseteq F^{(1)}(f^{-1}) M_0 \supseteq F^{(0)}(f^{-1}) M_0.$$
	Secondly, observe that 
	\[\bigcup_{e = 0}^\infty \Phi^{e*}(F^{(0)} (f^{-1}) M_0) = \bigcup_{e = 0}^\infty F^{(e)}(f^{-1}) \Phi^{e*}(M_0) = M[1/f]. \qedhere \]
\end{proof}

By using the \v{C}ech complex, the results from the previous sections tell us that the finiteness properties of localizations transfer over to local cohomology.

\begin{proposition} \label{prop-localCoh-F-mod}
	Let $R$ be a ring with {\upshape CL-FFRT} and fix notation as in Setup \ref{setup-FFRT}. Let $M$ be a $\cD^{(0)} \la \Phi \ra$-module, $I \sq R$ be an ideal and $i \geq 0$ be an integer. The local cohomology module $H^i_I(M)$ admits a left $\cD^{(0)} \la \Phi \ra$-structure. Moreover, if $M$ is finitely generated unit then so is $H^i_I(M)$.
\end{proposition}
\begin{proof}
	Fix generators $I = (f_1, \dots , f_s)$ for the ideal $I$, and consider the \v{C}ech complex $\check C^\bullet( \ul f; M)$ on these generators. Recall that all terms in this complex are direct sums of localizations of $M$ at single elements, and that the differentials are (up to a sign) given by natural localization maps.

	By Proposition \ref{prop-localization-fg}, $\check C^\bullet( \ul f; M)$ is a complex of left $\cD^{(0)} \la \Phi \ra$-modules, and since $H^i_I(M)$ is the $i$-th cohomology of this complex, it follows that $H^i_I(M)$ is also a left $\cD^{(0)} \la \Phi \ra$-module.

	For the last claim, suppose that $M$ is finitely generated unit. By Propositions \ref{prop-localization-fg} and \ref{prop-localization-single-elt-unit} all terms of $\check C^\bullet( \ul f; M)$ are also finitely generated unit. It follows from Corollary \ref{cor-fgu-abelian} that its cohomology modules are also finitely generated unit.
\end{proof}

\begin{corollary} \label{cor-iterated-localCoh-F-mod}
	If $I_1, \dots , I_r \sq R$ are ideals and $i_1, \dots , i_r \geq 0$ are integers, then $H^{i_1}_{I_1} \circ \cdots \circ H^{i_r}_{I_r}(R)$ is a finitely generated unit left $\cD^{(0)} \la \Phi \ra$-module.
\end{corollary}

\begin{remark} [cf. {\cite[Ex. 2.1(iii)]{Lyu93}}]
	Note that the proof of Proposition \ref{prop-localCoh-F-mod} makes crucial use of the \v{C}ech complex, and it is not a priori clear that the $\cD^{(0)}\la \Phi \ra$-module structure of $H^i_I(M)$ is independent of the choice of generators for $I$. An alternative construction is given as follows: let $X := \Spec(R)$, let $\Z_X \mmod$ be the category of abelian sheaves on $X$ and $\Z\mmod$ be the category of abelian groups. Given an $R$-module $M$ we denote by $M^\sim$ the associated quasicoherent sheaf on $X$. If $M$ is a left $\cD^{(0)} \la \Phi \ra$-module and $G : \Z_X\mmod \to \Z\mmod$ is a covariant additive functor then we claim that $G(M^\sim)$ has a natural $\cD^{(0)} \la \Phi \ra$-module structure. Indeed, as observed above $\cD^{(0)} \la \Phi \ra$ acts compatibly on every localization of $M$, and hence on $M^\sim$. In other words, there is a ring homomorphism  $\cD^{(0)} \la \Phi \ra \to \End_{\Z_X}(M^\sim)$. Postcomposing with the natural map $\End_{\Z_X}(M^\sim) \to \End_\Z(G(M^{\sim}))$ we obtain a ring homomorphism $\cD^{(0)} \la \Phi \ra \to \End_\Z(G(M^\sim))$ and hence a $\cD^{(0)} \la \Phi \ra$-module action on $G(M^\sim)$. 
\end{remark}

We are now ready to prove the finiteness of associated primes of iterated local cohomology of $R$. More generally, we show the following:

\begin{theorem}
	Let $R$ be a ring with {\upshape CL-FFRT}, and fix notation as in Setup \ref{setup-FFRT}. Let $M$ be a finitely generated unit $\cD^{(0)} \la \Phi \ra$-module. Then, when viewed as an $R$-module, $M$ has a finite number of associated primes.
\end{theorem}
\begin{proof}
	We claim there is a finite filtration
	$$ 0 = M_0 \sq M_1 \sq \cdots \sq M_n = M$$
	by $\cD^{(0)} \la \Phi \ra$-submodules such that each composition factor $M_i / M_{i - 1}$ has a single associated prime $\fp_i \sq R$ when viewed as an $R$-module. Note that from this claim it follows that $\Ass_R(M) \sq \{\fp_1, \dots , \fp_n\}$, which proves the theorem. 

	To build this chain, pick $M_0 = 0$. Given $i \geq 0$, if $M = M_i$ then the chain stops, and otherwise pick a prime $\fp_{i+1} \sq R$ which is maximal among the associated primes of $M / M_i$. It follows that $\fp_{i + 1}$ is the unique associated prime of the $\cD^{(0)} \la \Phi \ra$-submodule $H^0_{\fp_{i + 1}} (M / M_i) \sq M / M_i$. We define $M_{i + 1} \sq M$ to be the unique $\cD^{(0)} \la \Phi \ra$-submodule containing $M_i$ for which
	$$H^0_{\fp_{i + 1}} (M / M_i) = M_{i + 1} / M_i \sq M / M_i.$$
	This process produces an increasing chain of $\cD^{(0)} \la \Phi \ra$-submodules of $M$. Since $M$ is Noetherian (Corollary \ref{cor-fgu-abelian}), this chain must eventually stabilize to $M$, and hence the process stops after a finite number of steps.
\end{proof}

\begin{corollary} \label{cor-iterated-finite-ass}
	Let $R$ be a ring with {\upshape CL-FFRT}, let $I_1, \dots , I_r \sq R$ be ideals and $i_1, \dots , i_r \geq 0$ be integers. The iterated local cohomology module $H^{i_1}_{I_1} \circ \cdots \circ H^{i_r}_{I_r}(R)$ has a finite number of associated primes. 
\end{corollary}

Our next goal is to show that whenever $R$ has CL-FFRT and $f \in R$ is a nonzerodivisor the local cohomology modules of $R / fR$ have closed support. In the regular case this was proved simultaneously and independently by Hochster-N\'{u}\~{n}ez-Betancourt \cite{HNB17} and Katzman-Zhang \cite{KZ18}; our proof closely mimics the former.

Given a commutative ring $R$ and an $R$-module $M$ we denote the support of $M$ by $\Supp_R(M) \sq \Spec(R)$; that is, $\Supp_R(M)$ is the collection of prime ideals $\fp \sq R$ for which $M_\fp \neq 0$. 

\begin{lemma} \label{lemma-fgF-closed-supp}
	Let $R$ be a ring with {\upshape CL-FFRT}, and fix notation as in Setup \ref{setup-FFRT}. Let $M$ be a finitely generated $\cD^{(0)} \la \Phi \ra$-module. Then $M$ has closed support when viewed as an $R$-module.
\end{lemma}
\begin{proof}
	Fix generators $u_1, \dots, u_r \in M$ over $\cD^{(0)} \la \Phi \ra$, and let $M_0$ be the $R$-module they generate. We claim that $M$ and $M_0$ have the same support. Since $M_0 \sq M$ it is clear that the support of $M_0$ is contained in the support of $M$. For the other inclusion observe that, since $M_0$ generates $M$ as a $\cD^{(0)} \la \Phi \ra$-module, the natural map
	$$\cD^{(0)} \la \Phi \ra \otimes_R M_0 \to M$$
	is surjective. It follows that for every prime ideal $\fp \sq R$ the natural map
	$$\cD^{(0)} \la \Phi \ra \otimes_R (M_0)_\fp \to M_\fp$$
	is surjective. In particular, if $(M_0)_\fp = 0$ then $M_\fp = 0$, which proves the claim. 

	We thus have that $\Supp_R(M) = \Supp_R(M_0)$, and since $M_0$ is a finitely generated $R$-module it has closed support.
\end{proof}

\begin{lemma} \label{lemma-nilpt-support}
	Let $R$ be a commutative ring, let $f \in R$ be an element, and let $M$ be an $R$-module on which $f$ is nilpotent. Then $\Supp_R(M) = \Supp_R(M / fM)$. 
\end{lemma}
\begin{proof}
	It is clear that $\Supp_R(M / fM) \sq \Supp_R(M)$. For the other inclusion, fix an integer $n \gg 0$ such that $f^n M = 0$, and suppose $\fp \sq R$ is a prime for which $(M / fM)_\fp = 0$. We then get $M_\fp = f M_\fp$, and therefore $M_\fp = f M_\fp = f^2 M_\fp = \cdots = f^n M_\fp = 0$, hence $M_\fp = 0$. 
\end{proof}

\begin{theorem} \label{thm-closed-supp}
	Suppose $R$ has {\upshape CL-FFRT}, and fix notation as in Setup \ref{setup-FFRT}. Let $f \in R$ be a nonzerodivisor, $I \sq R$ be an ideal, and $i \geq 0$ be an integer. Then the local cohomology module $H^i_I(R / fR)$ has closed support.
\end{theorem}
\begin{proof}
	The short exact sequence
	$$
	\begin{tikzcd}
		0 \arrow[r] & R \arrow[r, "f"] & R \arrow[r] & R / fR \arrow[r] & 0
	\end{tikzcd}
	$$
	induces a long exact sequence in local cohomology, a portion of which looks like
	$$
	\begin{tikzcd}
		\cdots \arrow[r] &  H^i_I(R) \arrow[r, "f"] & H^i_I(R) \arrow[r] & H^i_I(R / fR) \arrow[r] & H^{i+1}_I(R) \arrow[r] & \cdots 
	\end{tikzcd}
	$$
	This induces a short exact sequence of the form
	$$
	\begin{tikzcd}
		0 \arrow[r] & 	H \big/ fH \arrow[r] & H^i_I(R / fR) \arrow[r] & Q \arrow[r] & 0,
	\end{tikzcd}
	$$
	where $H := H^i_I(R)$ and $Q \sq H^{i+1}_I(R)$ is a submodule.  Since $H^{i+1}_I(R)$ has a finite number of associated primes (Corollary \ref{cor-iterated-finite-ass}), so does $Q$, and therefore $Q$ has closed support. We claim that $H / f H$ has closed support; since $\Supp_R(H^i_I(R / fR)) = \Supp_R(Q) \cup \Supp_R (H / f H )$ this will prove the result. 

	Recall that $H$ is finitely generated over $\cD^{(0)} \la \Phi \ra$ (Corollary \ref{cor-iterated-localCoh-F-mod}). Since $F^{(0)}(f) H = f^{p^a} H$ is a $\cD^{(0)} \la \Phi \ra$-submodule, the quotient $H / f^{p^a} H$ acquires a $\cD^{(0)} \la \Phi \ra$-module structure, and the finite generation of $H$ tells us that $H / f^{p^a} H$ is finitely generated over $\cD^{(0)} \la \Phi \ra$. In particular, $\Supp_R(H / f^{p^a} H)$ is closed by Lemma \ref{lemma-fgF-closed-supp}. But, by Lemma \ref{lemma-nilpt-support}, we have $\Supp_R(H / fH) = \Supp_R(H / f^{p^a} H)$, and hence the claim follows. 
\end{proof}

\newcommand{\etalchar}[1]{$^{#1}$}


\begin{thebibliography}{{\`A}MHNB17}

\bibitem[AF92]{Anderson-Fuller-92}
F.~W. Anderson and K.~R. Fuller.
\newblock {\em Rings and categories of modules}, volume~13 of {\em Graduate
  Texts in Mathematics}.
\newblock Springer-Verlag, New York, second edition, 1992.

\bibitem[AK19]{Alhazmy-Katzman}
K.~Alhazmy and M.~Katzman.
\newblock F{FRT} properties of hypersurfaces and their {$F$}-signatures.
\newblock {\em J. Algebra Appl.}, 18(11):1950215, 15, 2019.

\bibitem[Alh17]{Alhazmy-FFRT}
K.~Alhazmy.
\newblock {\em On the {F}inite {F}-representation type and {F}-signature of
  hypersurfaces}.
\newblock PhD thesis, University of Sheffield, November 2017.

\bibitem[{\`A}MBL05]{AMBL}
J.~{\`A}lvarez-{M}ontaner, M.~Blickle, and G.~Lyubeznik.
\newblock Generators of {$D$}-modules in positive characteristic.
\newblock {\em Math. Res. Lett.}, 12(4):459--473, 2005.

\bibitem[{\`{A}}MHJ{\etalchar{+}}21]{AMHJNBTW21}
J.~{\`{A}}lvarez~Montaner, D.~J. Hern{\'{a}}ndez, J.~Jeffries,
  L.~N{\'{u}}{\~{n}}ez-Betancourt, P.~Teixeira, and E.~E. Witt.
\newblock Bernstein's inequality and holonomicity for certain singular rings,
  2021.
\newblock https://arxiv.org/abs/2103.02986.

\bibitem[{\`A}MHNB17]{AMHNB}
J.~{\`A}lvarez-{M}ontaner, C.~Huneke, and L.~N{\'u}{\~n}ez-Betancourt.
\newblock {$D$}-modules, {B}ernstein-{S}ato polynomials and {$F$}-invariants of
  direct summands.
\newblock {\em Adv. Math.}, 321:298--325, 2017.

\bibitem[BBL{\etalchar{+}}14]{BBLSZ}
B.~Bhatt, M.~Blickle, G.~Lyubeznik, A.~K. Singh, and W.~Zhang.
\newblock Local cohomology modules of a smooth {$\Bbb{Z}$}-algebra have
  finitely many associated primes.
\newblock {\em Invent. Math.}, 197(3):509--519, 2014.

\bibitem[Ber00]{BerthelotII}
P.~Berthelot.
\newblock {$D$}-modules arithm\'{e}tiques. {II}. {D}escente par {F}robenius.
\newblock {\em M\'{e}m. Soc. Math. Fr. (N.S.)}, (81):vi+136, 2000.

\bibitem[Ber12]{Berthelot-divided}
P.~Berthelot.
\newblock A note on {F}robenius divided modules in mixed characteristics.
\newblock {\em Bull. Soc. Math. France}, 140(3):441--458, 2012.

\bibitem[BF00]{BF00}
M.~P. Brodmann and A.~L. Faghani.
\newblock A finiteness result for associated primes of local cohomology
  modules.
\newblock {\em Proc. Amer. Math. Soc.}, 128(10):2851--2853, 2000.

\bibitem[BL17]{BL17}
B.~Bhatt and J.~Lurie.
\newblock A {R}iemann-{H}ilbert correspondence in positive characteristic,
  2017.
\newblock {https://arxiv.org/abs/1711.04148}.

\bibitem[Bli01]{Blickle-thesis}
M.~Blickle.
\newblock {\em The {I}ntersection {H}omology {D}-{M}odule in {F}inite
  {C}haracteristic}.
\newblock PhD thesis, University of Michigan, 2001.

\bibitem[B{\"o}g95]{Bogvad}
R.~B{\"o}gvad.
\newblock Some results on {$D$}-modules on {B}orel varieties in characteristic
  {$p>0$}.
\newblock {\em J. Algebra}, 173(3):638--667, 1995.

\bibitem[BRS00]{BRS00}
M.~Brodmann, C.~Rotthaus, and R.~Y. Sharp.
\newblock On annihilators and associated primes of local cohomology modules.
\newblock {\em J. Pure Appl. Algebra}, 153(3):197--227, 2000.

\bibitem[Cha74]{Chase-homdim}
S.~U. Chase.
\newblock On the homological dimension of algebras of differential operators.
\newblock {\em Comm. Algebra}, 1:351--363, 1974.

\bibitem[DQ20]{DQ20}
H.~Dao and P.~H. Quy.
\newblock On the associated primes of local cohomology.
\newblock {\em Nagoya Math. J.}, 237:1--9, 2020.

\bibitem[EK04]{EK04}
M.~Emerton and M.~Kisin.
\newblock The {R}iemann-{H}ilbert correspondence for unit {$F$}-crystals.
\newblock {\em Ast\'{e}risque}, (293):vi+257, 2004.

\bibitem[Gro65]{EGAIV}
A.~Grothendieck.
\newblock {\'E}l\'ements de {G}\'eom\'etrie {A}lg\'ebrique : {IV}. \'{E}tude
  locale des sch\'emas et des morphismes de sch\'emas, {S}econde partie.
\newblock {\em Publications Math\'ematiques de l'IH\'ES}, 24:5--231, 1965.

\bibitem[Haa88]{Haastert-invdir}
B.~Haastert.
\newblock On direct and inverse images of {${D}$}-modules in prime
  characteristic.
\newblock {\em Manuscripta Math.}, 62(3):341--354, 1988.

\bibitem[Har15]{Hara15}
N.~Hara.
\newblock Looking out for {F}robenius summands on a blown-up surface of
  {$\Bbb{P}^2$}.
\newblock {\em Illinois J. Math.}, 59(1):115--142, 2015.

\bibitem[Har70]{Hartshorne-inf-socle}
R.~Hartshorne.
\newblock Affine duality and cofiniteness.
\newblock {\em Invent. Math.}, 9:145--164, 1969/70.

\bibitem[HKM09]{HKM09}
C.~Huneke, D.~Katz, and T.~Marley.
\newblock On the support of local cohomology.
\newblock {\em Journal of Algebra}, 322(9):3194--3211, 2009.
\newblock Special Issue in Honor of Paul Roberts.

\bibitem[HNB17]{HNB17}
M.~Hochster and L.~N{\'{u}}{\~{n}}ez-Betancourt.
\newblock Support of local cohomology modules over hypersurfaces and rings with
  {FFRT}.
\newblock {\em Math. Res. Lett.}, 24(2):401--420, 2017.

\bibitem[HO20]{Hara-Ohkawa}
N.~Hara and R.~Ohkawa.
\newblock The {FFRT} property of two-dimensional normal graded rings and
  orbifold curves.
\newblock {\em Adv. Math.}, 370:107215, 37, 2020.

\bibitem[Hov01]{Hovey-wide}
M.~Hovey.
\newblock Classifying subcategories of modules.
\newblock {\em Trans. Amer. Math. Soc.}, 353(8):3181--3191, 2001.

\bibitem[HS93]{Huneke-Sharp}
C.~L. Huneke and R.~Y. Sharp.
\newblock Bass numbers of local cohomology modules.
\newblock {\em Trans. Amer. Math. Soc.}, 339(2):765--779, 1993.

\bibitem[Kat70]{Katz70}
N.~M. Katz.
\newblock Nilpotent connections and the monodromy theorem: {A}pplications of a
  result of {T}urrittin.
\newblock {\em Inst. Hautes \'{E}tudes Sci. Publ. Math.}, (39):175--232, 1970.

\bibitem[Kat02]{Katzman-counterexample}
M.~Katzman.
\newblock An example of an infinite set of associated primes of a local
  cohomology module.
\newblock {\em J. Algebra}, 252(1):161--166, 2002.

\bibitem[Kat06]{Katzman06}
M.~Katzman.
\newblock The support of top graded local cohomology modules.
\newblock In {\em Commutative algebra}, volume 244 of {\em Lect. Notes Pure
  Appl. Math.}, pages 165--174. Chapman \& Hall/CRC, Boca Raton, FL, 2006.

\bibitem[Kha10]{Khashyarmanesh10}
K.~Khashyarmanesh.
\newblock On the support of local cohomology modules and filter regular
  sequences.
\newblock {\em J. Commut. Algebra}, 2(2):177--185, 2010.

\bibitem[KZ18]{KZ18}
M.~Katzman and W.~Zhang.
\newblock The support of local cohomology modules.
\newblock {\em Int. Math. Res. Not. IMRN}, (23):7137--7155, 2018.

\bibitem[Lew22]{Lewis22}
M.~A. Lewis.
\newblock The local cohomology of a parameter ideal with respect to an
  arbitrary ideal.
\newblock {\em J. Algebra}, 589:82--104, 2022.

\bibitem[LW12]{LeuschkeWiegand-CM}
G.~J. Leuschke and R.~Wiegand.
\newblock {\em Cohen-{M}acaulay representations}, volume 181 of {\em
  Mathematical Surveys and Monographs}.
\newblock American Mathematical Society, Providence, RI, 2012.

\bibitem[Lyu93]{Lyu93}
G.~Lyubeznik.
\newblock Finiteness properties of local cohomology modules (an application of
  {$D$}-modules to commutative algebra).
\newblock {\em Invent. Math.}, 113(1):41--55, 1993.

\bibitem[Lyu97]{Lyubeznik-Fmod}
G.~Lyubeznik.
\newblock {$F$}-modules: applications to local cohomology and {$D$}-modules in
  characteristic {$p>0$}.
\newblock {\em J. Reine Angew. Math.}, 491:65--130, 1997.

\bibitem[Lyu00a]{Lyubeznik-charfree}
G.~Lyubeznik.
\newblock Finiteness properties of local cohomology modules: a
  characteristic-free approach.
\newblock {\em J. Pure Appl. Algebra}, 151(1):43--50, 2000.

\bibitem[Lyu00b]{Lyubeznik-unramified}
G.~Lyubeznik.
\newblock Finiteness properties of local cohomology modules for regular local
  rings of mixed characteristic: the unramified case.
\newblock volume~28, pages 5867--5882. 2000.
\newblock Special issue in honor of Robin Hartshorne.

\bibitem[Mal22]{Mallory22}
D.~Mallory.
\newblock Finite {$F$}-representation type for homogeneous coordinate rings of
  non-{F}ano varieties, 2022.
\newblock https://arxiv.org/abs/2207.08966.

\bibitem[Mar01]{Marley01}
T.~Marley.
\newblock The associated primes of local cohomology modules over rings of small
  dimension.
\newblock {\em Manuscripta Math.}, 104(4):519--525, 2001.

\bibitem[MV04]{MV-inf-socle}
T.~Marley and J.~C. Vassilev.
\newblock Local cohomology modules with infinite dimensional socles.
\newblock {\em Proc. Amer. Math. Soc.}, 132(12):3485--3490, 2004.

\bibitem[NB12]{NB12}
L.~N{\'{u}}{\~{n}}ez-Betancourt.
\newblock Local cohomology properties of direct summands.
\newblock {\em J. Pure Appl. Algebra}, 216(10):2137--2140, 2012.

\bibitem[NB13a]{NB13}
L.~N{\'{u}}{\~{n}}ez-Betancourt.
\newblock Local cohomology modules of polynomial or power series rings over
  rings of small dimension.
\newblock {\em Illinois J. Math.}, 57(1):279--294, 2013.

\bibitem[NB13b]{NB13-2}
L.~N{\'{u}}{\~{n}}ez-Betancourt.
\newblock On certain rings of differentiable type and finiteness properties of
  local cohomology.
\newblock {\em J. Algebra}, 379:1--10, 2013.

\bibitem[Rob12]{Robbins12}
H.~Robbins.
\newblock Associated primes of local cohomology and {$S_2$}-ification.
\newblock {\em J. Pure Appl. Algebra}, 216(3):519--523, 2012.

\bibitem[Rob14]{Robbins14}
H.~Robbins.
\newblock Associated primes of local cohomology after adjoining indeterminates.
\newblock {\em J. Pure Appl. Algebra}, 218(11):2072--2080, 2014.

\bibitem[Rob16]{Robbins16}
H.~Robbins.
\newblock Associated primes of local cohomology after adjoining indeterminates.
  {P}art 2: {T}he general case.
\newblock {\em J. Commut. Algebra}, 8(4):589--598, 2016.

\bibitem[R{\c{S}}05]{RS05}
C.~Rotthaus and L.~M. {\c{S}}ega.
\newblock Some properties of graded local cohomology modules.
\newblock {\em J. Algebra}, 283(1):232--247, 2005.

\bibitem[R{\v{S}}VdB19]{RSVdB19}
T.~Raedschelders, {\v{S}}.~{\v{S}}penko, and M.~Van~den Bergh.
\newblock The {F}robenius morphism in invariant theory.
\newblock {\em Adv. Math.}, 348:183--254, 2019.

\bibitem[R{\v{S}}VdB22]{RSVdB22}
T.~Raedschelders, {\v{S}}.~{\v{S}}penko, and M.~Van~den Bergh.
\newblock The {F}robenius morphism in invariant theory {II}.
\newblock {\em Adv. Math.}, 410:Paper No. 108587, 2022.

\bibitem[Shi17]{Shibuta17}
T.~Shibuta.
\newblock Affine semigroup rings are of finite {F}-representation type.
\newblock {\em Comm. Algebra}, 45(12):5465--5470, 2017.

\bibitem[Sin00]{Singh-counterexample}
A.~K. Singh.
\newblock {$p$}-torsion elements in local cohomology modules.
\newblock {\em Math. Res. Lett.}, 7(2-3):165--176, 2000.

\bibitem[Smi86]{SmithSP-diffOps}
S.~P. Smith.
\newblock Differential operators on the affine and projective lines in
  characteristic {$p>0$}.
\newblock In {\em S\'{e}minaire d'alg\`ebre {P}aul {D}ubreil et {M}arie-{P}aule
  {M}alliavin, 37\`eme ann\'{e}e ({P}aris, 1985)}, volume 1220 of {\em Lecture
  Notes in Math.}, pages 157--177. Springer, Berlin, 1986.

\bibitem[Smi87]{SmithSP-gdim}
S.~P. Smith.
\newblock The global homological dimension of the ring of differential
  operators on a nonsingular variety over a field of positive characteristic.
\newblock {\em J. Algebra}, 107(1):98--105, 1987.

\bibitem[SS04]{SS-counterexample}
A.~K. Singh and I.~Swanson.
\newblock Associated primes of local cohomology modules and of {F}robenius
  powers.
\newblock {\em Int. Math. Res. Not.}, (33):1703--1733, 2004.

\bibitem[SVdB97]{SVdB97}
K.~E. Smith and M.~Van~den Bergh.
\newblock Simplicity of rings of differential operators in prime
  characteristic.
\newblock {\em Proceedings of the London Mathematical Society}, 75(1):32–62,
  1997.

\bibitem[Tak08]{Takahashi-classifying}
R.~Takahashi.
\newblock Classifying subcategories of modules over a commutative {N}oetherian
  ring.
\newblock {\em J. Lond. Math. Soc. (2)}, 78(3):767--782, 2008.

\bibitem[TT08]{TT08}
S.~Takagi and R.~Takahashi.
\newblock {$D$}-modules over rings with finite {$F$}-representation type.
\newblock {\em Math. Res. Lett.}, 15(3):563--581, 2008.

\bibitem[Yao05]{Yao-FFRT}
Y.~Yao.
\newblock Modules with finite {$F$}-representation type.
\newblock {\em J. London Math. Soc. (2)}, 72(1):53--72, 2005.

\bibitem[Yek92]{Yek}
A.~Yekutieli.
\newblock An explicit construction of the {G}rothendieck residue complex.
\newblock {\em Ast\'{e}risque}, (208):127, 1992.
\newblock With an appendix by Pramathanath Sastry.

\end{thebibliography}
\end{document}